\documentclass[a4paper,12pt,reqno]{amsart}

\usepackage{amssymb,datetime,caption,comment,dsfont,empheq,enumerate,graphicx,manfnt,mathabx,mathrsfs,mathtools,psfrag,relsize,tikz,upgreek,xcolor,xparse}

\usepackage{url}

\usepackage[subnum]{cases}

\usepackage[normalem]{ulem}
\providecolor{added}{rgb}{0,0,1}
\providecolor{deleted}{rgb}{1,0,0}

\definecolor{labelkey}{rgb}{0.142, 0.812, 0.522}
\definecolor{refkey}{rgb}{0.0, 0., 0.}

\DeclareSymbolFont{euleroperators}{U}{eur}{m}{n}
\SetSymbolFont{euleroperators}{bold}{U}{eur}{b}{n}

\DeclareRobustCommand\nlab{k}

\newcommand{\mytag}[2]{\tag*{(\theequation #1)#2}}

\makeatletter
\renewcommand{\operator@font}{\mathgroup\symeuleroperators}
\makeatother

\usepackage{fancyvrb}

\makeatletter
\let\c@equation\c@figure
\makeatother

\makeatletter
\def\@seccntformat#1{%
  \protect\textup{\protect\@secnumfont
    \ifnum\pdfstrcmp{subsection}{#1}=0 \bfseries\fi% subsection # in \bfseries
    \csname the#1\endcsname
    \protect\@secnumpunct
  }%
}
\makeatother

\makeatletter
\def\@tocline#1#2#3#4#5#6#7{\relax
   \ifnum #1>\c@tocdepth % then omit
   \else
     \par \addpenalty\@secpenalty\addvspace{#2}%
     \begingroup \hyphenpenalty\@M
     \@ifempty{#4}{%
       \@tempdima\csname r@tocindent\number#1\endcsname\relax
     }{%
       \@tempdima#4\relax
     }%
     \parindent\z@ \leftskip#3\relax \advance\leftskip\@tempdima\relax
     \rightskip\@pnumwidth plus4em \parfillskip-\@pnumwidth
     #5\leavevmode\hskip-\@tempdima #6\nobreak\relax
     \ifnum#1<0\hfill\else\dotfill\fi\hbox to\@pnumwidth{\@tocpagenum{#7}}\par
     \nobreak
     \endgroup
   \fi}
\makeatother

\NewDocumentCommand{\bn}{sO{}m}{%
  {\IfBooleanTF{#1}
    {\oldnormaux{\left|}{\right|}{#3}}
    {\oldnormaux{#2|}{#2|}{#3}}}
}
\makeatletter
\newcommand{\oldnormaux}[3]{\mathpalette\oldnormaux@i{{#1}{#2}{#3}}}
\newcommand{\oldnormaux@i}[2]{\oldnormaux@ii#1#2}
\newcommand{\oldnormaux@ii}[4]{%
  \sbox\z@{$\m@th#1#2#4#3$}%
  \sbox\tw@{$\m@th\|$}%
  \mathopen{\hbox to\wd\tw@{\hss\vrule height \ht\z@ depth \dp\z@ width .2\wd\tw@\hss}}%
  #4
  \mathclose{\hbox to\wd\tw@{\hss\vrule height \ht\z@ depth \dp\z@ width .2\wd\tw@\hss}}%
}
\makeatother

\numberwithin{equation}{section}

\renewcommand{\thefigure}{\arabic{section}.\arabic{figure}}

\renewcommand{\thefigure}{\ifnum\value{section}>0 \arabic{section}.\fi\arabic{figure}}

\renewcommand{\theequation}{\ifnum\value{section}>0 \arabic{section}.\fi\arabic{equation}}

\newtheorem{thm}[equation]{Theorem}

\newtheorem{cor}[equation]{Corollary}
\newtheorem{lem}[equation]{Lemma}
\newtheorem{prp}[equation]{Proposition}

\newcommand{\thistheoremname}{}

\newtheorem*{genericthm*}{\thistheoremname}
\newenvironment{namedthm*}[1]
  {\renewcommand{\thistheoremname}{#1}%
   \begin{genericthm*}}
  {\end{genericthm*}}

\theoremstyle{definition}
\newtheorem{dfn}[equation]{Definition}
\newtheorem{ex}[equation]{Example}

\theoremstyle{remark}
\newtheorem{rem}[equation]{Remark}

\newcommand{\thmref}[1]{Theorem~\ref{#1}}
\newcommand{\prpref}[1]{Proposition~\ref{#1}}
\newcommand{\lemref}[1]{Lemma~\ref{#1}}
\newcommand{\corref}[1]{Corollary~\ref{#1}}
\newcommand{\dfnref}[1]{Definition~\ref{#1}}
\newcommand{\remref}[1]{Remark~\ref{#1}}
\newcommand{\exref}[1]{Example~\ref{#1}}
\newcommand{\figref}[1]{Figure~\ref{#1}}
\newcommand{\secref}[1]{Section~\ref{#1}}

\newcommand\bend[1]{\raisebox{1.5mm}{{\tiny\dbend}\;}{\textit{\textsf{#1}}}}

\DeclareMathOperator{\bnd}{bnd}
\DeclareMathOperator{\C}{C}
\DeclareMathOperator{\card}{card}
\DeclareMathOperator{\gr}{gr}
\DeclareMathOperator{\sgr}{sgr}
\DeclareMathOperator{\supp}{supp}

\newcommand{\Ab}{\boldsymbol{A}}
\newcommand{\Ag}{\mathfrak A}
\newcommand{\al}{\alpha}

\newcommand{\be}{\beta}
\newcommand{\Bg}{\mathfrak B}
\newcommand{\BS}{\operatorname{BS}}

\newcommand{\Cc}{\mathcal C}

\newcommand{\Dc}{\mathcal D}
\newcommand{\De}{\Delta}
\newcommand{\de}{\delta}

\newcommand{\Eg}{\mathfrak E}

\newcommand{\ep}{\varepsilon}

\newcommand{\FC}{\operatorname{FC}}
\newcommand{\Fc}{\mathcal F}
\newcommand{\free}{\mathop{\Huge \mathlarger{\mathlarger{*}}}}
\newcommand{\Fo}{\operatorname{F}}

\newcommand{\ga}{\gamma}
\newcommand{\gab}{\boldsymbol{\gamma}}

\newcommand{\ib}{\;\raisebox{0.25ex}{$\scriptstyle\bullet$}\;}
\newcommand{\Ic}{\mathcal I}

\newcommand{\ka}{\kappa}

\newcommand{\La}{\Lambda}
\newcommand{\la}{\lambda}

\newcommand{\mapstoto}{\mathop{\,\sim\joinrel\rightsquigarrow\,}}
\newcommand*{\mi}{\leavevmode\hphantom{0}\llap{\settowidth{\dimen0}{$0$}\resizebox{1.1\dimen0}{\height}{$-$}}}
\newcommand{\mr}[1]{\overset{\smash{\lower.1em\hbox{$\,\mathsmaller{\mathsmaller{\bullet}}$}}}{#1}}
\newcommand{\ms}{\mathstrut}
\newcommand\myp{\mkern-.0mu\raise0.4ex\hbox{$\scriptstyle\prime$}}

\newcommand{\N}{{\!\!\boldsymbol{\#}}}

\newcommand{\NN}{\mathbb N}
\newcommand{\nnu}{\boldsymbol{\upnu}}
\newcommand{\ns}{\smalltriangleleft}

\newcommand{\om}{\omega}
\newcommand{\ov}{\overline}

\newcommand{\p}{\partial}
\newcommand{\Pb}{\mathbf P}

\newcommand{\ph}{\varphi}
\newcommand{\Ps}{\p_\mu G}

\newcommand{\RA}{\operatorname{R}_{\operatorname{amen}}}

\newcommand{\RR}{\mathbb R}

\newcommand{\Sc}{\mathcal S}
\newcommand{\sd}{\rightthreetimes}
\newcommand{\Sf}{\mathfrak S}
\newcommand{\Sfb}{\boldsymbol{\mathfrak S}}
\newcommand{\sh}{\sharp}
\newcommand{\Si}{\Sigma}
\newcommand{\si}{\sigma}
\newcommand{\Sib}{\boldsymbol{\Sigma}}

\renewcommand{\sl}{\overrightarrow}
\newcommand{\sm}{\,\setminus\,}
\newcommand{\sr}{\overleftarrow}
\newcommand{\Stab}{\operatorname{Stab}}

\newcommand{\Tc}{\mathcal T}
\newcommand{\tha}{\theta}

\newcommand{\Uc}{\mathcal U}

\newcommand{\vn}{\varnothing}

\newcommand{\wh}{\widehat}
\newcommand{\wt}{\widetilde}

\newcommand{\xb}{\boldsymbol{x}}

\newcommand{\yb}{\boldsymbol{y}}

\newcommand{\ze}{\zeta}
\newcommand{\ZZ}{\mathbb Z}

\begin{document}

\author{Anna Erschler}
\address{A.E.: C.N.R.S., Ecole Normale Superieur, PSL Research University,  45 rue d'Ulm, Paris, France}
\email{anna.erschler@ens.fr}

\author{Vadim A.\ Kaimanovich}
\address{V.K.: Department of Mathematics and Statistics, University of
Ottawa, 150 Louis Pasteur, Ottawa ON, K1N 6N5, Canada}
\email{vkaimano@uottawa.ca, vadim.kaimanovich@gmail.com}

\title[Arboreal structures on groups]{Arboreal structures on groups \\ and the associated boundaries}

\thanks{The work of the first named author was supported by the ERC grant GroIsRan and by the ANR grant MALIN}

\begin{abstract}
For any countable group with infinite conjugacy classes we construct a family of forests on the group. For each of them there is a random walk on the group with the property that its sample paths almost surely converge to the geometric boundary of the forest in a way that resembles the simple random walks on trees. It allows us to identify the Poisson boundary of the random walk with the boundary of the forest and to show that the group action on the Poisson boundary is free (which, in particular, implies non-triviality of the Poisson boundary). As a consequence we obtain that any countable group carries a random walk such that the stabilizer of almost every point of the Poisson boundary coincides with the hyper-FC-centre of the group, and, more generally, we characterize all normal subgroups which can serve as the pointwise stabilizer of the Poisson boundary of a random walk on a given countable group. Our work is a development of a recent result of Frisch - Hartman - Tamuz - Vahidi Ferdowsi who proved that any group which is not hyper-FC-central admits a measure with a non-trivial Poisson boundary.
\end{abstract}

\vspace*{-1.cm}
\maketitle

{\footnotesize\sffamily\tableofcontents}

\thispagestyle{empty}

\vfill\eject

\phantom{z}

\setlength{\headsep}{1cm}

\vspace{.5cm}

\section*{Introduction}

The Poisson boundary of a Markov chain is a measure space that describes the stochastically significant behaviour of the chain at infinity. It provides an integral representation of the bounded harmonic functions of the chain, so that, in particular, the Poisson boundary is trivial if and only if the chain has no non-constant bounded harmonic functions (the Liouville property). The origins of this concept go back to Blackwell \cite{Blackwell55}, Feller \cite{Feller56}, Doob \cite{Doob59}, and, in the special case of random walks on groups, to Dynkin -- Malyutov \cite{Dynkin-Malutov61} and Furstenberg \cite{Furstenberg63, Furstenberg67, Furstenberg71} (for random walks one sometimes uses the term Poisson -- Furstenberg boundary as well, see  \cite{Vershik00, Erschler10}). The Poisson boundary is related to the asymptotic geometry of the underlying group; see Bartholdi -- Virag \cite{Bartholdi-Virag05} (as well as \cite{Amir-Angel-MatteBon-Virag16} and the references therein for the latest developments in this direction) concerning the use of the Liouville property for establishing amenability of groups, and \cite{Erschler04a, Erschler-Zheng18p} for the applications to growth. It also plays a role in the rigidity theory \cite{Furstenberg71, Kaimanovich-Masur96, Farb-Masur98}.

It is known that a countable group admits a random walk with a trivial Poisson boundary if and only if the group is amenable \cite{Furstenberg73, Kaimanovich-Vershik83,Rosenblatt81}. On the other hand, the Poisson boundary is trivial for \emph{any} irreducible random walk on an abelian group; this result goes back to Blackwell \cite{Blackwell55} (also see \cite{Doob-Snell-Williamson60, Choquet-Deny60}), and was further extended to nilpotent groups \cite{Dynkin-Malutov61, Margulis66}, and ultimately to hyper-FC-central groups\footnotemark\; \cite{Lin-Zaidenberg98, Jaworski04}.

\footnotetext{\;
The union of all finite conjugacy classes of a group $G$ is its normal subgroup called the {\bf FC-centre} and denoted $\FC(G)$. If any non-identity element of $G$ has infinitely many conjugates, i.e., if $\FC(G)$ is trivial, then $G$ is said to be a group with infinite conjugacy classes (or, an {\bf ICC group} in short). The  {\bf hyper-FC-centre} $\FC_{\lim}(G)$ is the limit of the transfinite upper FC-series of the group, and it is the minimal normal subgroup of $G$ with the property that the associated quotient group is ICC. A group is called {\bf hyper-FC-central} if it coincides with its hyper-FC-centre. A finitely generated group is hyper-FC-central if and only if it is virtually nilpotent.}

In a striking recent result Frisch~-- Hartman~-- Tamuz~-- Vahidi Ferdowsi \cite{Frisch-Hartman-Tamuz-Ferdowsi18} showed that, on the other hand, any countable group which is \emph{not} hyper-FC-central does admit irreducible measures with a non-trivial boundary; in particular, this is the case for any finitely generated group of exponential growth which solves a conjecture of Vershik and the second named author stated in \cite{Kaimanovich-Vershik83}. Our Theorem A below is a development of their result.

\begin{namedthm*}{Theorem A}[\;=\;\thmref{thm:main}]
For any countable ICC group $G$ there exist a probability measure $\mu$ on~$G$ (which can be chosen to be symmetric, non-degenerate, and of finite entropy) and a locally finite forest $\Fc$ with the vertex set $G$ such that:
\begin{enumerate}[{\rm (i)}]
\item
almost all sample paths of the random walk $(G,\mu)$ converge to the boundary $\p\Fc$ of the forest $\Fc$;
\item
almost every sample path visits all but a finite number of points of any geodesic ray in $\Fc$ ending at the limit point of the path;
\item
the arising quotient map from the Poisson boundary $\p_\mu G$ to the boundary $\p\Fc$ endowed with the family of the resulting hitting distributions is an isomorphism;
\item
the action of the group $G$ on the Poisson boundary $\p_\mu G$ is free \textup{(mod 0)}.
\end{enumerate}
\end{namedthm*}

In contrast to \cite{Frisch-Hartman-Tamuz-Ferdowsi18}, where \emph{quantitative} estimates on the total variance distance between the convolution powers of $\mu$ and their translates were used to claim the non-triviality of the boundary for amenable ICC groups, Theorem A is of a \emph{qualitative} nature. Its base is a novel \emph{arboreal structure} present on \emph{any} ICC group $G$ (no matter, amenable or non-amenable), which is a forest of rooted trees with the vertex set $G$. It is in terms of this forest
that we exhibit a non-trivial behaviour of sample paths at infinity in claim~(i) and further
explicitly identify the Poisson boundary in claim (iii). Another difference between our approach and that of \cite{Frisch-Hartman-Tamuz-Ferdowsi18} is similar to the difference between the strong and the weak laws of large numbers: the latter is a property of the sequence of convolution powers, whereas the former is a property of individual sample paths.

Claim (ii) is a strong convergence property. In the case of simple random walks on free groups and, more generally, on trees, it was used already by Dynkin -- Malyutov \cite{Dynkin-Malutov61} and Cartier \cite{Cartier72}, respectively, for an identification of the topological Martin boundary with the space of ends. However, it is quite a rare phenomenon otherwise, even within the class of simple random walks.
Claim (ii) means that any ICC group has a random walk with an infinitely supported step distribution which, in a sense, behaves like the simple random walk on a tree.

A basic question about the Poisson boundary is to provide its complete description (or identification) for a given Markov chain. Even in such a fundamental example as a free group, for which the convergence to the boundary of the group is known for any random walk (which essentially goes back to \cite[Section 4]{Furstenberg67}), it is still an open question whether this convergence completely describes the Poisson boundary.  The strong convergence property of claim (ii) almost automatically implies the identification of the Poisson boundary in claim (iii) without any restrictions imposed by the conditional entropy criterion \cite{Kaimanovich00a} (the current standard tool for identifying the Poisson boundary).

Nearly 50 years ago Furstenberg noticed that \emph{the advantage of the Poisson boundary is that perturbing the measure $\mu$ does not appear to radically change the nature of the Poisson boundary}
\cite[p.\ 25]{Furstenberg71}. Indeed, in all currently known examples of an explicit identification of the Poisson boundary its underlying space remains the same for all step distributions $\mu$ from a sufficiently wide class (although the hitting distributions and their measure classes do depend on the choice of $\mu$), see \cite{Kaimanovich00a, Ledrappier01, Lyons-Peres15, Maher-Tiozzo18p} and the references therein. Theorem A provides many measures~$\mu$ on the same ICC group~$G$ such that their Poisson boundaries do not have a common underlying space (the forests, in terms of which the boundary is described, significantly depend on $\mu$). This absence of universality seems to be the first example of this kind.

It is well-known that for any random walk $(G,\mu)$ the action of the group $G$ on the Poisson boundary $\p_\mu G$ is ergodic, and therefore this action is non-trivial whenever the Poisson boundary is not a singleton. It is claimed in \cite[Remark 2.7]{Frisch-Hartman-Tamuz-Ferdowsi18} that if $G$ is an amenable ICC group, then for any $g\in G\sm\{e\}$ there is a measure $\mu=\mu_g$ such that the action of~$g$ on $\p_\mu G$ is non-trivial. Our claim (iv) states that on any ICC group there is a random walk with the property that the group action on the Poisson boundary is not only faithful (i.e., any non-identity element acts non-trivially), but also (mod 0) free (i.e., the stabilizer subgroup of almost every boundary point is trivial). It might be interesting to compare claim (iv) with the equivalence of the $C^*$-simplicity of a countable group $G$ and the freeness of the action on its topological Furstenberg boundary recently proved in \cite[Theorem~1.5]{Kalantar-Kennedy17} (also see \cite[Theorem~3.1]{Breuillard-Kalantar-Kennedy-Ozawa17}) as well as with the work of Hartman -- Kalantar \cite{Hartman-Kalantar17p} on $C^*$-simple probability measures on groups.

Talking about a general (not necessarily ICC) group $G$, the action of the hyper-FC-centre $\FC_{\lim}(G)$ on the Poisson boundary $\p_\mu G$ is trivial for any non-degenerate step distribution $\mu$ (i.e., such that $\supp\mu$ generates $G$ as a semigroup, \prpref{prp:fc}; this need not be true for degenerate measures, see \secref{sec:ex}). Therefore, Theorem A applied to the quotient ICC group
$G/\FC_{\lim}(G)$ implies the existence of a probability measure $\mu$ on $G$ such that the Poisson boundary $\p_\mu G$ can be explicitly described, any non-hyper-FC-central element of $G$ acts on the $\p_\mu G$ non-trivially, and, moreover, the stabilizers of almost all points of the Poisson boundary coincide with $\FC_{\lim}(G)$.

We can actually completely characterize the subgroups of $G$ which arise as the pointwise stabilizers of the Poisson boundary of the non-degenerate random walks on $G$:

\begin{namedthm*}{Theorem B}[\;=\;\thmref{thm:stab}]
For any countable group $G$ the following conditions on a subgroup $H\subset G$ are equivalent:
\begin{enumerate}[{\rm (i)}]
\item
there exists a non-degenerate probability measure $\mu$ on $G$ such that the stabilizer of almost every point of the Poisson boundary $\p_\mu G$ is $H$;
\item
there exists a non-degenerate probability measure $\mu$ on $G$ such that the pointwise stabilizer of the Poisson boundary $\p_\mu G$ is $H$ (i.e., $H$ is the intersection of almost all point stabilizers);
\item
$H$ is an amenable normal subgroup of $G$ with the property that the quotient $G/H$ is either an ICC group or the trivial group.
\end{enumerate}
\end{namedthm*}

In the particular case when $G$ is an ICC group and $H=\{e\}$ this is Theorem A(iv), and in the case when $G$ is amenable and $H=G$ this is the aforementioned result from \cite{Rosenblatt81, Kaimanovich-Vershik83} on the existence of random walks with the trivial Poisson boundary on any amenable group. The proof of the main implication (iii)$\implies$(i) consists in applying Theorem A(iv) to the quotient group $\ov G=G/H$ to obtain a non-degenerate probability measure $\ov\mu$ on $\ov G$ such that the action of $\ov G$ on the Poisson boundary $\p_{\ov\mu}{\ov G}$ is (mod 0) free. Because of the amenability of $H$ the measure $\mu$ can then be lifted to a measure $\mu$ with the property (i) by \cite[Theorem 2]{Kaimanovich02a}.

The measure $\mu$ in Theorem B need not have finite entropy in contrast to Theorem A, as, for instance, the amenable wreath product $G=\ZZ^3\wr\ZZ_2$ has no finite entropy measures with a non-trivial boundary \cite{Erschler04}.

\bigskip

\textbf{Arboreal structures. An outline of the proof of Theorem A.}

\medskip

\textbf{1.}
An \textsf{arboreal structure} on a countable group $G$ can be viewed as a mapping from a forest with $G$-labelled edges to $G$. A particular case of this notion, which is at the core of the proof of Theorem~A, is when this mapping is bijective, so that the vertex set of the forest can be identified with the group. Its construction for an arbitrary ICC group (see Sections \ref{sec:ladders} and \ref{sec:sicc}) is inspired by the notions of switching and super-switching elements introduced and applied by Frisch~-- Tamuz -- Vahidi Ferdowsi \cite{Frisch-Tamuz-Ferdowsi18} and Frisch~-- Hartman~-- Tamuz~-- Vahidi Ferdowsi \cite{Frisch-Hartman-Tamuz-Ferdowsi18}, respectively. The version of this construction described below is simplified (with symmetric switching sets and no filling sets, in terms of \dfnref{dfn:lad}) compared to the one we actually use in the proof of Theorem~A, but it is still sufficient for dealing with finitely generated groups.

Given a sequence of symmetric subsets $\Si_n\subset G$ and a \textsf{gauge function} $\la$, we define a \textsf{spike decomposition} (\dfnref{dfn:spike}) of an element $g\in G$ as its presentation as the product
%$$
\begin{equation} \tag{$\bigstar$}
g = \sl g \cdot \wh g \cdot \sr g
\end{equation}
%$$
of a \textsf{spike} $\wh g\in\Si_n$ (the number $n$ is called the \textsf{height} of the decomposition)
by at most $\la(n)$ multipliers (on either side) from the preceding sets $\Si_1,\Si_2,\dots,\Si_{n-1}$:
$$
\sl g,\sr g\in \De_n^{\la(n)},\;\quad\text{where}\quad \De_n=\Si_1 \cup \Si_2 \cup \dots \cup \Si_{n-1} \;.
$$
The associated \textsf{despiking graph} is the graph with the vertex set $G$ whose edges are all pairs $(g,\sl g)$, where $\sl g$ is the prefix of a decomposition ($\bigstar$) of an element $g\in G$.

The key role in endowing the group~$G$ with an arboreal structure, which is essential for our proof of Theorem A, is played by the \emph{uniqueness} of spike decomposition ($\bigstar$). If this is the case for a group element $g\in G$, then we can define its \textsf{height} $\bn g$ as the height of the unique decomposition of $g$. By $\Uc\subset G$ we denote the set of all $g\in G$ which are \textsf{unspiked} (admit no spike decomposition), and put for them $\bn g=0$. We say that a sequence $(\Si_n)$ is a $\la$\textsf{-ladder} (\dfnref{dfn:lad}) if any $g\in G$ has at most one spike decomposition, and, additionally, $\bn{\sl g}<\bn g$ for any $g\in G\sm\Uc$. If $(\Si_n)$ is a $\la$-ladder, then the associated despiking graph is a forest denoted by~$\Fc$. Each of the trees of $\Fc$ contains a unique unspiked vertex (its root), and $g\to\sl g$ is the orientation of the edges of $\Fc$ towards the corresponding roots (\prpref{prp:root}).

The forest $\Fc$ can also be described recursively, by ``growing'' its trees from the set $\Uc$ of the unspiked elements of the group (\secref{sec:Markov}). In this description the geodesic rays of $\Fc$ issued from a root $g_0\in\Uc$ are the sequences
$$
g_0, g_1, \dots, g_k, \dots \qquad\text{with}\qquad g_{k-1}=\sl{g_k} \;,
$$
i.e., their entries are the products
$$
g_k = g_{k-1} \si_k h_k = g_0\si_1 h_1\si_2 h_2\dots\si_k h_k \;,
$$
see \figref{fig:forest}, which satisfy the following conditions:
{\setlength\leftmargini{1.5cm}
\begin{itemize}
\item[(a)]
$\si_k\in\Si_{n_k}$ for an increasing sequence $n_k=\bn{g_k}$;
\item[(b--c)]
$g_{k-1}$ and $h_k$ are in the ball $\De_{n_k}^{\la(n_k)}$ for all $k\ge 1$.
\end{itemize}

\captionsetup[figure]{font=small}
\begin{figure}[h]
\begin{center}
\psfrag{f}[l][r]{$\si_k h_k$}
\psfrag{e}[l][cl]{$g_{k-1}$}
\psfrag{a}[r][b]{$g_0$}
\psfrag{c}[cl][cl]{$g_1$}
\psfrag{b}[cl][cl]{$\si_1h_1$}
\psfrag{g}[l][c]{$g_k$}
\includegraphics[scale=.6]{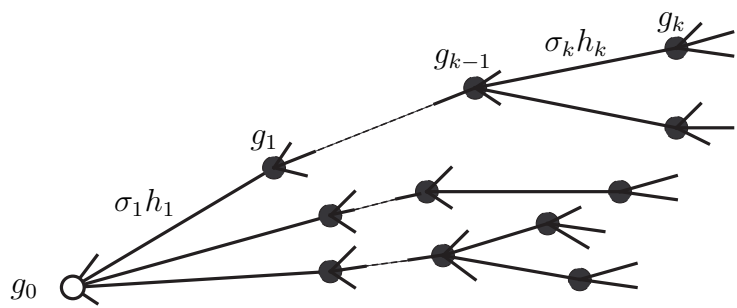}
\end{center}
\caption{\footnotesize
An arboreal structure for a rapidly growing gauge function $\la$. The edge labels satisfy conditions (a)  $\si_k \in \Si_{n_k}$ for a strictly increasing integer sequence~$n_k$, and \newline (b) $h_k\in\bigl(\Si_1\cup\dots\cup\Si_{n_k-1}\bigr)^{\la(n_k)}$.}
\label{fig:forest}
\end{figure}

If the function $\la$ grows fast enough, then the fact that $g_{k-1}\in \De_{n_k}^{\la(n_k)}$ for $k>1$ in condition (b--c) actually follows just from the restrictions on $h_k$, so that instead of (b--c) it is enough to require
\begin{itemize}
\item[(b)]
$h_k \in \De_{n_k}^{\la(n_k)}$ for all $k\ge 1$,
\item[(c$_1$)]
$g_0\in \De_{n_1}^{\la(n_1)}$.
\end{itemize}
Therefore (\prpref{prp:Markov}), in this situation the geodesic rays of the forest $\Fc$ are Markovian with the transition rules (a), (b) for any $k\ge 2$ and the initial conditions
\begin{itemize}
\item[(b$_1$--c$_1$)]
$g_0$ and $h_1$ are in the ball $\De_{n_1}^{\la(n_1)}$;\smallskip
\item[(d)]
the root $g_0$ is unspiked (in other words, $g_0$ can not be the root of a tree satisfying conditions (a) and (b--c)).
\end{itemize}
}

The boundary $\p\Fc$ of the forest $\Fc$ is defined in the usual way as the disjoint union of the boundaries of its trees.  Although it is not preserved by the action of the group~$G$, the definition of a ladder easily implies that the maximal $G$-invariant subset $\p\Fc_G\subset\p\Fc$ (\dfnref{dfn:gi}) is non-empty, and, moreover, the action of the group $G$ on $\p\Fc_G$ is~free. An explicit non-empty $G$-invariant subset $\p_\sh\Fc\subset\p\Fc$ corresponds to the set of
geodesic rays $g_0, g_1, \dots$ with the property that
$\la\left(\bn{ g_k }\right) - \left| g_{k-1} \right|_{\bn{ g_k }}  \to \infty$ (cf. condition (b--c) above), where $|g|_n$ denotes the length of $g\in G$ with respect to the set $\De_n=\Si_1 \cup \Si_2 \cup \dots \cup \Si_{n-1}$ (\prpref{prp:free}).

\medskip

\textbf{2.} For proving Theorem A we only need that the measure $\mu$ agree, in the sense to be specified below, with a ladder on the group $G$. We begin by choosing a probability measure $p=(p_j)$ on $\ZZ_+$ with the property that the record values $R_k$ of almost every sequence of independent $p$-distributed random variables are eventually simple (meaning that $R_k$ does not occur a second time before the next record value $R_{k+1}>R_k$ is achieved). There is a well-known explicit necessary and sufficient condition for that (\lemref{lem:bsw}) satisfied, for instance, for any distribution $p$ with a polynomial decay.

Then we choose a rapidly growing function $\la$ (\lemref{lem:psi}). A sufficient condition on its growth is given by the bound $T_{k+1} \le \la (R_k)$ for the record time $T_{k+1}$ in terms of the previous record value $R_k=X_{T_k}$, which is required to be eventually satisfied for almost every independent $p$-distributed sequence $(X_n)$.

Let now $(\Si_n)$ be a $\la$-ladder, and $\mu$ be a probability measure on the group $G$ such that $\mu(\Si_n)=p_n$. Take a sequence of independent $\mu$-distributed increments $x_n$ of the random walk $(G,\mu)$, and let $T_k$ and $R_k=\bn{x_{T_k}}$ be, respectively, the record times and the associated record values of the sequence $\bn{x_n}$ (we remind that $\bn \si=n$ for $\si\in\Si_n$, so that $\bn{x_n}$ are independent and $p$-distributed). If a particular record value $R_k$ is simple, then
the position $y_n = x_1 x_2 \dots x_n$ of the random walk at any time $n$ with $T_k\le n<T_{k+1}$ is spiked with the spike $\wh{\,y_n} = x_{T_k}$, and $\sl{y_n} = x_1 x_2 \dots x_{T_k-1} = y_{T_k-1}$
(i.e., $y_n$ and $y_{T_k-1}$ are neighbours in the forest $\Fc$). In particular,
$\sl{y_{T_{k+1}-1}} = y_{T_k-1}$, so that the subsequence $\left(y_{T_k-1}\right)$ is a geodesic ray for sufficiently large $k$, and the whole sample path $(y_n)$ converges to the limit point of this geodesic ray (\thmref{thm:poi}), see \figref{fig:forestwalk} for an illustration.

\captionsetup[figure]{font=small}
\begin{figure}[h]
\begin{center}
\includegraphics[scale=.5]{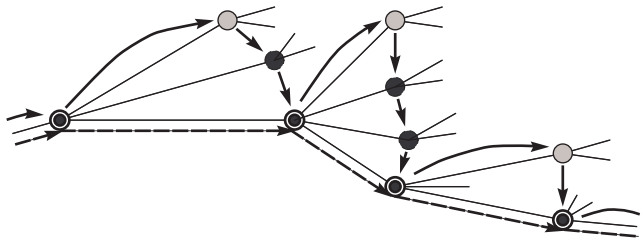}
\end{center}
\captionsetup{singlelinecheck=off}
\caption[.]{\footnotesize Boundary convergence. The solid arrows represent the random walk transitions, whereas the dashed arrows represent the limit geodesic ray. The circled bullets are the points $y_{T_k-1}$, and the greyed out circles are the points $y_{T_k}$.}
\label{fig:forestwalk}
\end{figure}

While the forest $\Fc$ is not locally finite, we also describe its locally finite \textsf{constrained subforest} (\dfnref{dfn:constr}) and show, in \thmref{thm:poi}(v), that the sample paths of the random walk almost surely converge to the boundary of that subforest as well.

\medskip

\textbf{3.} In order to complete the proof of Theorem A it remains to show that any ICC group has a $\la$-ladder for any gauge function $\la$. Given a subset $Z\subset G$, we say that a set $\Si\subset G\sm\{e\}$ is $Z$\textsf{-switching} if
$$
\si z =  z' \si' \iff z=z'=e,\; \si=\si'
$$
whenever $z,z'\in Z$ and $\si,\si'\in \Si$, i.e., if the product sets $\Si Z$ and $Z \Si$ have no common elements other than the ones arising for ``trivial'' reasons (\dfnref{dfn:sset}, this is a somewhat more general form of the notions introduced in \cite{Frisch-Tamuz-Ferdowsi18, Frisch-Hartman-Tamuz-Ferdowsi18}). The key property relating switching sets and ladders is that a sequence $(\Si_n)$, in which each $\Si_n$ is switching for the set $(\Si_1\cup\Si_2\cup\dots\Si_{n-1})^{5\la(n)}$, is a $\la$-ladder (\lemref{lem:5la}). Therefore, for making sure that the measure $\mu$ in Theorem A can be chosen symmetric, the above switching sets~$\Si_n$ have to be symmetric as well.

In the terminology of \cite{Frisch-Tamuz-Ferdowsi18, Frisch-Hartman-Tamuz-Ferdowsi18} the switching sets of the form $\Si=\{\si\}$ and $\Si=\{\si,\si^{-1}\}$ correspond to the \emph{switching} and the \emph{superswitching elements} $\si$, respectively. The existence of an infinite number of switching elements for any finite subset in an ICC group $G$ was established in \cite[Claim 3.3]{Frisch-Tamuz-Ferdowsi18}. Under the additional assumption that $G$ is \emph{amenable} this
was extended to superswitching elements in \cite[Proposition~2.4]{Frisch-Hartman-Tamuz-Ferdowsi18}. We eliminate this assumption and show that there are infinitely many superswitching elements for any finite subset in \emph{any} ICC group (\prpref{prp:ss}).

Now, given a finitely generated ICC group, we can take for $\Si_1$ any finite symmetric generating set and recursively obtain a symmetric $\la$-ladder for any gauge function $\la$, which completes the proof of Theorem A in the finitely generated case. For an infinitely generated ICC group the whole construction (including the definition of ladders) has to be modified in order to make sure that the support of the resulting measure $\mu$ generates the whole group. This is done by allowing the presence of an exhausting sequence of \emph{filling sets} (\dfnref{dfn:scale}), which then has to be taken into account in the aforementioned arguments.

\medskip

\textbf{Acknowledgement.} We are grateful to Todor Tsankov for indicating a reference to Neumann's Lemma \cite[(4.2)]{Neumann54}.

\section{Ladders, the associated forests, and their boundaries} \label{sec:ladders}

\subsection{Spike decomposition and ladders}

\begin{dfn} \label{dfn:scale}
A \textsf{scale} on a countable group $G$ is a triple
$$
\La=(\la,\Sib,\Ab)
$$
which consists of an unbounded non-decreasing \textsf{gauge function}
$$
\la:\NN\to\NN
$$
and two sequences
$$
\Sib = (\Si_1,\Si_2,\dots) \;, \qquad
\Ab = (A_0,A_1,\dots)
$$
of subsets of $G$ called the \textsf{spiking} and the \textsf{filling} sequences, respectively. All the spiking sets $\Si_n$ are required to be non-empty and pairwise disjoint, whereas the filling sets $A_n$ are allowed to be empty.
\end{dfn}

By
%$$
\begin{equation} \label{eq:de}
\De_n = \bigcup_{i=1}^{n-1} \Si_i^{\pm 1} \; \cup \; \bigcup_{i=0}^{n-1} A_i^{\pm 1} \;,\qquad n\ge 1 \;,
\end{equation}
%$$
we denote the union of the spiking sets $\Si_i$, the filling sets $A_i$, and their inverses with the indices \emph{strictly} less than $n$. By passing, if necessary, to a smaller subgroup, we shall assume without any loss of generality that
%$$
\begin{equation} \label{eq:asi}
\textsf{the set \; $\De_\infty = \bigcup_{n=1}^\infty \De_n$ \; generates the group $G$ \;,}
\end{equation}
%$$
i.e., equivalently, that the sequence of sets $\De_n^{\la(n)}$ exhausts the whole group~$G$. For the sake of technical simplicity we shall also assume that
%$$
\begin{equation} \label{eq:a0e}
\textsf{the set $A_0$ is non-empty and contains the group identity $e$\;.}
\end{equation}
%$$

By
%$$
\begin{equation} \label{eq:length}
|g|_\al = \min \bigl\{k\ge 0: g \in \De_\al^k \bigr\} \;, \qquad \al=1,2,\dots,\infty \;,
\end{equation}
%$$
we denote the \textsf{length} of a group element $g\in G$ with respect to the corresponding set $\De_\al$. Although $|g|_\infty<\infty$ for any $g\in G$ by assumption \eqref{eq:asi}, individual sets $\De_n$ with $n<\infty$ need not generate the whole group $G$, so that some of the lengths $|g|_n$ may well be infinite. However, since $\De_n\nearrow\De_\infty$,
%$$
\begin{equation} \label{eq:ni}
|g|_n \searrow |g|_\infty < \infty \qquad\forall\,g\in G \;.
\end{equation}
%$$

\begin{dfn} \label{dfn:spike}
A \textsf{spike decomposition} of a group element $g\in G$ with respect to a scale $\La=(\la,\Sib,\Ab)$ is its presentation as a product
$$
g = \sl g \cdot \wh g \cdot \sr g \;,
$$
where
\begin{enumerate}[{\rm (i)}]
\item
the \textsf{spike} $\wh g$ belongs to $\Si_n$ for a certain $n>0$ called the \textsf{height} of the decomposition;
\item
the \textsf{prefix} $\sl g$ and the \textsf{postfix} $\sr g$ are both products of at most $\la(n)$ multipliers from the union $\De_n$ \eqref{eq:de} of the sets $\Si_k^{\pm 1}, A_k^{\pm 1}$ with the indices $k<n$, i.e.,
%$$
\begin{equation} \label{eq:lan}
|\sl g|_n, |\sr g|_n \le \la(n) \;.
\end{equation}
%$$
\end{enumerate}
If $g\in G$ does not admit any spike decompositions, we call it \textsf{unspiked} and denote the set of all unspiked elements by $\Uc=\Uc(\La)\subset G$. If $g\in G$ has a unique spike decomposition, the components $\sl g, \wh g ,\sr g$ of this decomposition are referred to as the \textsf{prefix}, the \textsf{spike}, and the \textsf{postfix} of $g$, respectively, and we use the notation $\bn g =n$ (provided $\wh g\in\Si_n$) for the \textsf{height} of $g$ (we put $\bn g=0$ for all $g\in\Uc$).
\end{dfn}

\begin{dfn} \label{dfn:forest}
The \textsf{despiking graph} of a scale $\La=(\la,\Sib,\Ab)$ is the graph with the vertex set $G$ whose edges are all pairs $(g,\sl g)$, where $\sl g$ is the prefix of a spike decomposition of an element $g\in G$.
\end{dfn}

\begin{dfn} \label{dfn:lad}
A scale $\La=(\la,\Sib,\Ab)$ is called a \textsf{ladder} if
\begin{enumerate}[{\rm (i)}]
\item
any $g\in G$ admits at most one spike decomposition, i.e., the spike decomposition is unique for any $g\in G\sm\Uc$;
\item
the resulting \textsf{prefix map}
%$$
\begin{equation} \label{eq:prmap}
\uppi:g\mapsto \sl g
\end{equation}
%$$
which assigns to any $g\in G\sm\Uc$ the prefix $\sl g$ of its spike decomposition has the property that
%$$
\begin{equation} \label{eq:ineq}
\bn{\sl g} < \bn g \qquad \forall\,g\in G\sm\Uc \;.
\end{equation}
%$$
\end{enumerate}
\end{dfn}

Thus, if $\La$ is a ladder, then for any $g\in G$ the preimage set
$$
\uppi^{-1}(g) = \left\{g\myp \in G : \sl{g\myp} = g \right\}
$$
consists of all neighbours $g\myp$ of $g$ in the despiking graph whose height $\bn{g\myp}$ is strictly greater than $\bn g$, whereas for any $g\in G\sm\Uc$ the prefix $\uppi(g)=\sl g$ is the only neighbour of $g$ whose height is strictly smaller than $\bn g$.

\begin{rem} \label{rem:escape}
Inequality \eqref{eq:ineq} from \dfnref{dfn:lad} is automatically satisfied when the spiking sets $\Si_n$ ``escape to infinity'' fast enough, more precisely, if
%$$
\begin{equation} \label{eq:si}
\Si_n \cap \De_n^{3\la(n)} = \vn \qquad \forall\, n\ge 1
\end{equation}
%$$
(cf. \lemref{lem:5la} below). Indeed, under condition \eqref{eq:si} inequalities \eqref{eq:lan} imply, by the triangle inequality, that
$$
|g|_{\bn g} > \la(\bn g) \qquad \forall\, g\in G\sm\Uc \;.
$$
Thus, if $\bn{\sl g} \ge \bn g$ for a certain $g\in G\sm\Uc$, then by the monotonicity of the lengths $|\cdot|_n$ \eqref{eq:ni} and of the gauge function $\la$
$$
|\sl g|_{\bn g} \ge |\sl g|_{\scriptsize\bn{\sl g}} > \la(\bn{\sl g}) \ge \la(\bn g) \;,
$$
so that $|\sl g|_{\bn g} > \la(\bn g)$ in contradiction to \dfnref{dfn:spike}(ii).
\end{rem}

\subsection{The ladder forest}

Let us remind that a \textsf{tree} is a connected graph with no cycles. A tree endowed with a reference vertex (\textsf{root}) is called \textsf{rooted}. A \textsf{forest} is a graph (not connected in general) with no cycles, i.e., a graph whose connected components are trees. A forest is \textsf{rooted} if each of its trees is rooted.

\begin{prp} \label{prp:root}
For any ladder $\La$ on a group $G$
\begin{enumerate}[{\rm (i)}]
\item
the associated despiking graph $\Fc=\Fc(\La)$ is a forest;
\item
any tree of the forest $\Fc$ contains a unique unspiked vertex;
\item
all vertices of the forest $\Fc$ are of infinite order.
\end{enumerate}
\end{prp}

By using claim (ii) above we can, from now on, root the trees from the \textsf{ladder forest} $\Fc=\Fc(\La)$ by choosing the corresponding unspiked vertices as their roots. We denote by $\Tc_u\in\Fc$ the tree rooted at $u\in\Uc\subset G$, so that $\Fc=\{\Tc_u\}_{u\in\Uc}$. If $g\in\Tc_u$, then we also say that $u$ is the \textsf{root} of~$g$.

\begin{proof}[Proof of \prpref{prp:root}]
\;(i)\; Let $g_0,g_1,\dots,g_n=g_0$ be a cycle without backtracking in~$\Fc$. If $g_i$ is a maximal height vertex from this cycle, then $\bn{g_i}$ is strictly greater than both $\bn{g_{i-1}}$ and $\bn{g_{i+1}}$, whence $g_{i-1}=g_{i+1}$, which yields a contradiction.

(ii)\; The connected component of any spiked element of $G$ contains an unspiked one (it is produced by an iterative application of the prefix map $\uppi:g\mapsto\sl g$ until one reaches $\Uc$; the iterative process does eventually stop by condition (ii) from \dfnref{dfn:lad}). On the other hand, no connected component of $\Fc$ can contain more than one unspiked element (otherwise, by taking a path without backtracking which joins two unspiked elements and looking at its maximal height vertex, one arrives at a contradiction in precisely the same way as in the proof of~(i)).

(iii)\; By the definition of the spike decomposition (\dfnref{dfn:spike}), for any $g\in G$ the preimage set $\uppi^{-1}(g)$ of the prefix map \eqref{eq:prmap} consists of all products $g\myp=g \si h$ with
$\si\in\bigcup_n \Si_n$ and
%$$
\begin{equation} \label{eq:sin}
|g|_{\bn \si}, |h|_{\bn \si} \le \la\left(\bn\si\right) \;.
\end{equation}
%$$
[Let us remind that since $\La$ is a ladder, the spike of any $\si\in\Si_n$ is $\si$ itself, so that $\bn\si=n$.] Since the gauge function $\la$ is non-decreasing unbounded, by \eqref{eq:ni}
$$
n_0(g) = \min \left\{ n: |g|_n\le\la(n)  \right\} < \infty \;,
$$
and the inequality $|g|_n\le\la(n)$ holds whenever $n\ge n_0$, so that the set
%$$
\begin{equation} \label{eq:uppi}
\uppi^{-1}(g) = \bigl\{ g\si h: \si\in\Si_n \;\text{with}\; n\ge n_0(g) , \;\text{and}\; h\in G \;\text{with}\; |h|_n\le \la(n) \bigr\}
\end{equation}
%$$
is infinite for any $g\in G$.
\end{proof}

\prpref{prp:root} has the following two ramifications.

\subsubsection{A Markov property of the ladder forests} \label{sec:Markov}

The description of the preimage sets of the prefix map $\uppi$ that was used in the proof of \prpref{prp:root}(iii) leads to a recursive construction of the ladder forest $\Fc=\Fc(\La)$, in which its trees $\Tc_u$ are ``grown'' from their roots $u\in\Uc$ by applying the consecutive transitions $g\to g\myp\in\uppi^{-1}(g)$. The geodesic rays of $\Fc$ issued from a root $u\in\Uc$ are, therefore, the sequences
$$
u = g_0, g_1, \dots, g_k, \dots
$$
with
$$
g_k = g_0 \si_1 h_1 \si_2 h_2 \dots \si_k h_k \;, \qquad \si_k \in \bigcup_{i=1}^\infty \Si_i \qquad\qquad \text{for} \; k\ge 1 \;,
$$
where each increment $\si_k h_k$ satisfies conditions \eqref{eq:sin} with respect to the previous product $g_{k-1}$, see \figref{fig:forest} from the Introduction, i.e.,

\begin{subequations} \label{eq:incr}
\begin{empheq}[left={\text{for any}\; k\ge 1 \quad \empheqlbrace\;}]{align}
& \bn{\si_{k+1}} > \bn{\si_k} = \bn{g_k} \;,
\mytag{a}{$_{\nlab+1}$} \label{eq:sii} \\
& \notag \\
& |h_k|_{\bn{\si_k}} \le \la\left(\bn{\si_k} \right) \;,
\mytag{b}{$_{\nlab}$} \label{eq:r} \\
& \notag \\
& |g_{k-1}|_{\bn{\si_k}} = |g_0 \si_1 h_1 \dots \si_{k-1} h_{k-1}|_{\bn{\si_k}} \le \la\left(\bn{\si_k}\right) \;,
\mytag{c}{$_{\nlab}$} \label{eq:l}
\end{empheq}
with the ``initial condition''
\begin{equation}
g_0 \in \Uc \mytag{d}{} \label{eq:g0}
\end{equation}
\end{subequations}

\medskip

If the function $\la$ grows fast enough, then conditions \ref{eq:l} with $k\ge 2$ are actually redundant as they follow from conditions \ref{eq:r} on the increments $h_k$. Indeed, by the subadditivity of the word lengths \eqref{eq:length} and by their monotonicity \eqref{eq:ni}, for any $k\ge 2$
$$
\begin{aligned}
|g_{k-1}|_{\bn{\si_k}}
&= |g_0 \si_1 h_1 \dots \si_{k-1} h_{k-1}|_{\bn{\si_k}} \\
&\le |g_0|_{\bn{\si_k}} + |\si_1|_{\bn{\si_k}} + |h_1|_{\bn{\si_k}} + \dots + |\si_{k-1}|_{\bn{\si_k}} + |h_{k-1}|_{\bn{\si_k}} \\
&= (k-1) + |g_0|_{\bn{\si_1}} + |h_1|_{\bn{\si_1}} + \dots + |h_{k-1}|_{\bn{\si_{k-1}}} \;.
\end{aligned}
$$
Therefore, if the function $\la$ satisfied the inequality
%$$
\begin{equation} \label{eq:fast}
\la(n+1) \ge n + 2 \bigl( \la(1) + \la(2) + \dots + \la(n) \bigr) \qquad\forall\,n\ge 1 \;,
\end{equation}
%$$
then for any $k\ge 2$ conditions \renewcommand\nlab{1}\ref{eq:l} and
\ref{eq:r}, \renewcommand\nlab{2}\ref{eq:r}, \dots, \renewcommand\nlab{k-1}\ref{eq:r} imply the inequality
$$
\begin{aligned}
|g_{k-1}|_{\bn{\si_k}}
&\le (k-1) + \la\left(\bn{\si_1}\right) + \la\left(\bn{\si_1}\right) + \dots + \la\left(\bn{\si_{k-1}}\right) \\
&\le \la\left(\bn{\si_{k-1}}+1\right)
\le \la\left(\bn{\si_k}\right) \;,
\end{aligned}
$$
i.e., condition \renewcommand\nlab{k}\ref{eq:l}. Thus, we have

\begin{prp} \label{prp:Markov}
If the gauge function $\la$ of a ladder $\La=(\la,\Sib,\Ab)$ satisfies the fast growth condition \eqref{eq:fast}, then the geodesic rays issued from any root $u\in\Uc$ of the ladder forest $\Fc=\Fc(\La)$ are the sequences
$$
u = g_0, g_1, \dots, g_k, \dots
$$
with
$$
g_k = g_0 \si_1 h_1 \si_2 h_2 \dots \si_k h_k \;, \qquad k\ge 1 \;,
$$
where $\si_k\in\bigcup_{i=1}^\infty \Si_i$ and $h_k\in G$ satisfy conditions \textup{\ref{eq:sii}, \ref{eq:r}}\, for all $k\ge 1$ and condition \renewcommand\nlab{1} \textup{\ref{eq:l}}. Therefore, given an initial segment $g_0,g_1,\dots,g_{k-1},\;k\ge 2$, the set of all admissible points  $g_k=g_{k-1}\si_k h_k$ depends on $g_{k-1}$ only (the \textsf{Markov property}), and the geodesic rays in $\Fc$ can be described by the
``transition rules'' \renewcommand\nlab{k}
\begin{subequations}\label{eq:Markov}
\begin{empheq}[left={\textup{for any}\; k\ge 2 \quad\empheqlbrace\;}]{align}
& \bn{\si_k} > \bn{ g_{k-1}} \;,
\mytag{a}{$_{\nlab}$} \label{eq:Markova} \\
\notag\\
& |h_k|_{\bn{\si_k}} \le \la\left(\bn{\si_k} \right)
\mytag{b}{$_{\nlab}$} \label{eq:Markovb} \;.
\end{empheq}
and the ``initial conditions'' \renewcommand\nlab{1}
\begin{empheq}[left={\empheqlbrace\;}]{align}
& |h_1|_{\bn{\si_1}} \le \la\left(\bn{\si_1} \right) \;,
\mytag{b}{$_{\nlab}$} \label{eq:Markovb1} \\
\notag\\
& |g_0|_{\bn{\si_1}} \le \la\left(\bn{\si_1}\right) \;.
\mytag{c}{$_{\nlab}$} \label{eq:0} \\
\notag \\
& g_0 \in \Uc \;.
\mytag{d}{$_{\phantom{k}}$} \label{eq:Markov0}
\end{empheq}
\end{subequations}
\end{prp}

\begin{rem}
Fast growth condition \eqref{eq:fast} is, for instance, satisfied for the function $\la(n)=4^n$, or, more generally, if $\la(n)\ge 4^n$ and the ratio $\la(n)/4^n$ is non-decreasing.
\end{rem}

\subsubsection{Constrained subforests}

If all the entries $\Si_i, A_i$ of a ladder $\La=(\la,\Sib,\Ab)$ are finite, then there is a \emph{locally finite} subforest of the ladder forest $\Fc=\Fc(\La)$ (with an infinite number of trees though) still suitable for our further purposes.

\begin{dfn} \label{dfn:constr}
Given a ladder $\La$ on a group $G$, the \textsf{constrained ladder forest} $\Fc_\ka=\Fc_\ka(\La)$ determined by a non-decreasing \textsf{constraining function} $\ka:\ZZ_+\to\ZZ_+$ is the graph with the vertex set $G$ whose non-oriented edges are the pairs $\left(g,\sl g\right)$, where $g$ is spiked, $\sl g$ is the prefix of its spike decomposition, and, additionally, $\bn g \le \ka\left( \bn{\sl g} \right)$.
\end{dfn}

In other words, $\Fc_\ka$ is obtained from the ladder forest $\Fc$ by erasing all edges $\left(g,\sl g\right)$ with $\bn g > \ka\left( \bn{ \sl g } \right)$. The description \eqref{eq:uppi} of the preimage sets of the prefix map \mbox{$\uppi:g\mapsto\sl g$} in the proof of \prpref{prp:root}(iii) then implies

\begin{lem} \label{lem:constr}
If a ladder $\La=(\la,\Sib,\Ab)$ is such that all its entries $\Si_i,A_i$ are finite, then the constrained ladder forest $\Fc_\ka(\La)$ is locally finite for any constraining function $\ka$.
\end{lem}

\begin{rem}
\prpref{prp:Markov} is applicable to the constrained forests as well. One just has to replace condition
\ref{eq:Markova} from the description of the Markov transitions with the condition
$$
\bn{ g_{k-1}} < \bn{\si_k} \le \ka(\bn{ g_{k-1}}) \;.
$$
\end{rem}

\subsection{Boundaries of trees and forests} \label{sec:bnd}

The \textsf{boundary} $\p\Tc$ of a rooted tree $(\Tc,o)$ can be identified with the set of geodesic rays in $\Tc$ issued from the root $o$. The boundary carries the usual topology of pointwise convergence, which is compact if and only if the tree is locally finite. A sequence of vertices $\tau_n\in\Tc$ \textsf{converges} to a boundary point $\tau_\infty\in\p\Tc$ if the geodesic segments $[o,\tau_n]$ pointwise converge to the geodesic ray representing $\tau_\infty$.

Let us remind that two sequences $\xb=(x_0,x_1,\dots)$ and $\xb'=(x'_0,x'_1,\dots)$ are called \textsf{asymptotically equivalent} (notation: $\xb\sim\xb'$) if there exist integers $n,n'\ge 0$ such that
%$$
\begin{equation} \label{eq:eq}
x_{n+i}=x'_{n'+i} \qquad \forall\, i\ge 0 \;.
\end{equation}
%$$
We denote the \textsf{asymptotic equivalence class} of a sequence $\xb$ by $[\xb]$. In the ``rootless'' description the boundary $\p\Tc$ of a tree $\Tc$ is defined as the set of the classes of asymptotic equivalence of geodesic rays in $\Tc$. It is the presence of a root $o\in\Tc$ that allows one to choose the geodesic ray $\ov\xi=[o,\xi)$ issued from the root $o$ as the ``canonical representative'' of each equivalence class $\xi\in\p\Tc$.

The \textsf{boundary $\p\Fc$ of a forest $\Fc$} is the disjoint union
$$
\p\Fc=\bigsqcup_{\Tc\in\Fc} \p\Tc
$$
of the boundaries of all trees $\Tc$ which constitute the forest. The topology of $\p\Fc$ is that of a disjoint union of the topological spaces $\p\Tc$. In the same way as for a single tree, by making a forest \textsf{rooted} (i.e., by rooting each tree of the forest) one can identify the boundary $\p\Fc$ with the set of all geodesic rays issued from these roots. A sequence of vertices $\tau_n$ \textsf{converges} to a boundary point $\tau_\infty\in\p\Fc$ if the sequence is eventually contained in a single tree $\Tc\in\Fc$, and $\tau_n\to\tau_\infty\in\p\Tc\subset\p\Fc$ in this tree. Given a boundary point $\xi\in\p\Fc$, the tree $\Tc\in\Fc$ whose boundary contains $\xi$ will be called the \textsf{tree of} $\xi$, and the root of $\Tc$ will be called the \textsf{root of} $\xi$.

Let us now assume that the vertex set of our forest $\Fc$ is a countable group $G$. We emphasize that we do not impose any \emph{a priori} conditions on the relationship between the graph structure of $\Fc$ and the group structure of $G$.

The group $G$ naturally acts by entry-wise translations
%$$
\begin{equation} \label{eq:action}
(g_0,g_1,\dots) \mapsto (g g_0,g g_1,\dots)
\end{equation}
%$$
on sequences $(g_0,g_1,\dots)$ in $G$ and on their classes of asymptotic equivalence. Therefore,
given a forest $\Fc$ on $G$, its boundary $\p\Fc$ (as a collection of the asymptotic equivalence classes of geodesic rays) is moved by any group element $g\in G$ to the boundary $\p g\Fc=g\p\Fc$ of the translated forest $g\Fc$. However, the boundary $\p\Fc$ need \emph{not} be preserved by the action of $G$, see \exref{ex:zz} and \exref{ex:free} below.

\begin{rem}\label{rem:stab}
It is straightforward to describe the \textsf{stabilizer subgroups}
$$
\Stab\xi = \{ g\in G: g\xi=\xi \} \;, \qquad \xi\in\p\Fc \;.
$$
Namely, given a countable group $G$, a forest $\Fc$ on $G$, and $g\in G,\xi\in\p\Fc$, one has $g\xi=\xi$ if and only for any ($\equiv$ a certain) geodesic ray $\gab=(\ga_0,\ga_1,\dots)$ converging to $\xi$ there is an integer $t$ such that $g\ga_n = \ga_{n+t}$ for all sufficiently large $n$.
\end{rem}

\begin{dfn} \label{dfn:gi}
The \textsf{$G$-invariant part} $\p_G\Fc$ of the boundary $\p\Fc$ of a forest $\Fc$ on a countable group $G$ is the intersection
$$
\p_G\Fc = \bigcap_{g\in G} g\p\Fc = \bigcap_{g\in G} \p g\Fc \;.
$$
In other words, $\p_G\Fc$ is the maximal subset of $\p\Fc$ invariant under the action of $G$ by translations, i.e., it consists of the equivalence classes of all geodesic rays $\gab=(\ga_0,\ga_1,\dots)$ in $\Fc$ with the property that any translate $g\gab,\; g\in G$, is asymptotically equivalent to a certain geodesic ray in $\Fc$ (which, of course, depends on $g$ and on $\gab$).
\end{dfn}

\begin{rem} \label{rem:bdry}
If the forest $\Fc$ is rooted, then, as we have already noticed, any boundary point $\xi\in\p\Fc$ can be uniquely represented by the geodesic ray $\gab=[o,\xi)=(\ga_0,\ga_1,\dots)$ issued from the corresponding root $o=\ga_0$. In these terms the $g$-translate $\xi'=g\xi$ belongs to $\p\Fc$ if and only if the translates $(g\ga_n,g\ga_{n+1},\dots)$ of the truncated rays $[\ga_n,\xi)=(\ga_n,\ga_{n+1},\dots)$ are also geodesic rays in $\Fc$ for all sufficiently large $n$. If this is the case, let $o'$ denote the root of $\xi'$, i.e., the root of the common tree $\Tc'$ of the vertices $g\ga_n,g\ga_{n+1},\dots$\;. Then the geodesics $[o',g\ga_n]$ stabilize, and their concatenation with the ray $(g\ga_n,g\ga_{n+1},\dots)$ coincides, for sufficiently large $n$, with the geodesic ray $\gab'=[o',\xi')$ joining the root $o'$ with the boundary point $\xi'=g\xi$, see \figref{fig:action}. Thus, the sequences of \textsf{increments} $\ga_n^{-1}\ga_{n+1}$ and ${\ga'_n}^{-1}\ga'_{n+1}$ along the rays $\gab=[o,\xi)$ and $\gab'=[o',g\xi)$ are asymptotically equivalent, and the action of $g$ amounts to a ``rewriting'' (with a possible length change) of a certain initial part of the sequence of increments $\bigl(\ga_n^{-1}\ga_{n+1}\bigr)$.
\end{rem}

\begin{figure}[h]
\begin{center}
\psfrag{a}[][]{$o'=\ga'_0$}
\psfrag{b}[][]{$\ga'_1$}
\psfrag{c}[][]{$\ga'_{n'-1}$}
\psfrag{d}[][]{$\ga'_{n'}=g\ga_n$}
\psfrag{e}[][]{$\qquad\ga'_{n'+1}=g\ga_{n+1}$}
\psfrag{f}[][r]{$\qquad\xi'=g\xi$}
\psfrag{g}[][]{$\p\Tc'$}
\psfrag{h}[][]{$g\ga_0$}
\psfrag{i}[][]{$g\ga_1$}
\psfrag{j}[][]{$g\ga_{n-1}$}
\psfrag{k}[][]{$o=\ga_0$}
\psfrag{l}[][]{$\ga_1$}
\psfrag{m}[][]{$\ga_{n-1}$}
\psfrag{n}[][]{$\ga_n$}
\psfrag{o}[][]{$\ga_{n+1}$}
\psfrag{p}[][]{$\xi$}
\psfrag{q}[][]{$\p\Tc$}
\includegraphics[scale=.6]{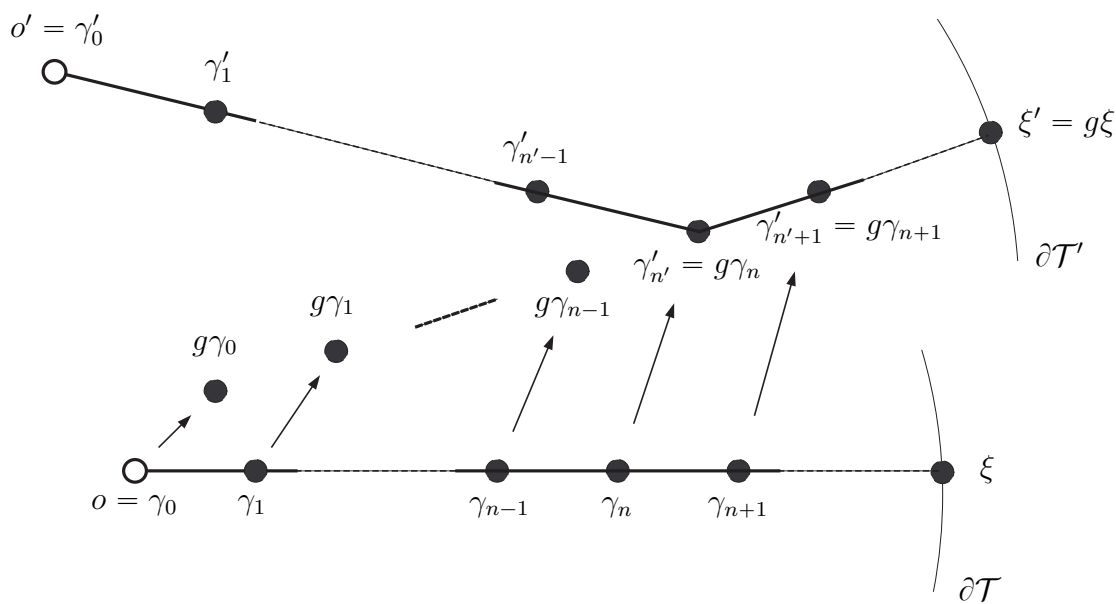}
\end{center}
\caption{The boundary action.}
\label{fig:action}
\end{figure}

The discussion from \remref{rem:bdry} easily implies that the subset $\p_G\Fc\subset\p\Fc$ is Borel. However, in general it need not be closed in $\p\Fc$, see \exref{ex:free}.

\begin{ex} \label{ex:zz}
The forest $\Fc$ on the group $G=\ZZ^2$ presented in \figref{fig:zz} on the left consists of the ``vertical'' bilateral geodesics $\Tc_i$ indexed with the integers $i\neq 0,1$ and of a 3-ended tree~$\Tc$ with the vertex set $\{(i,j): i=0,1,\;j\in\ZZ\}$. The boundary $\p\Tc_i$ of each $\Tc_i$ consists of two points $\xi_i^-,\xi_i^+$ represented by the geodesic rays
$\bigl( (i,0), (i,-1), (i,-2), \dots \bigr)$ and $\bigl( (i,0), (i,1), (i,2), \dots \bigr)$, respectively, whereas the boundary $\p\Tc$ of $\Tc$ consists of two points $\xi_0^-,\xi_1^-$ defined as above and a third point $\xi^+$ represented by the ``zigzag'' geodesic ray
$\bigl( (0,0), (1,1), (0,1), (1,2), (0,2), (1,3),\dots \bigr)$. The $\ZZ^2$-invariant part of the boundary of~$\Fc$ consists of a single $\ZZ^2$-orbit $\{\xi_i^-\}_{i\in\ZZ}$. A slight modification of this construction produces a forest on $\ZZ^2$, for which the $\ZZ^2$-invariant part of the boundary is empty, see the right side of \figref{fig:zz}.
\end{ex}

\begin{figure}[h]
\begin{center}
\includegraphics[scale=.4]{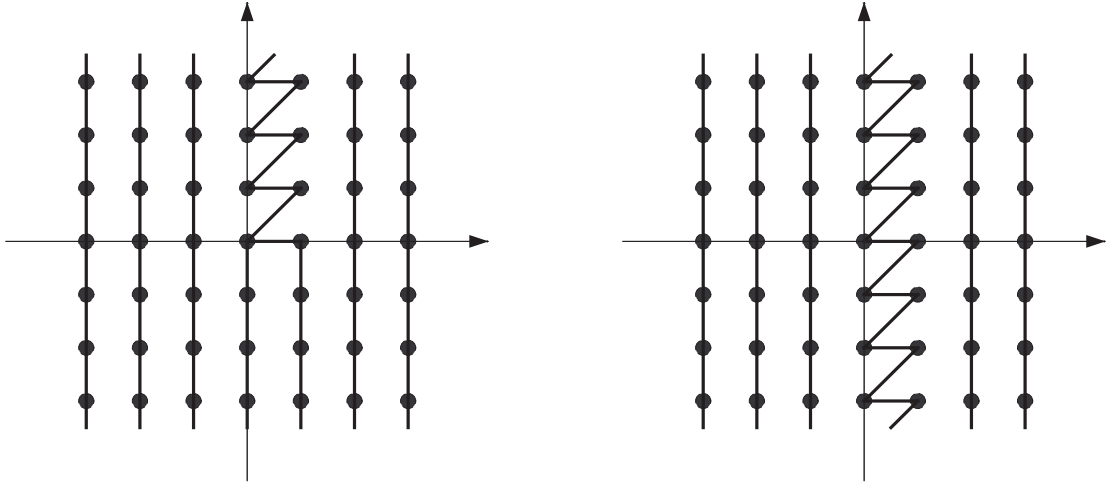}
\end{center}
\caption{Forests on $\ZZ^2$ whose boundaries are not preserved by the group action.}
\label{fig:zz}
\end{figure}

\begin{ex} \label{ex:free}
Let $G=\Fo_2=\langle a,b \rangle$ be the free group with the free generating set $\{a,b\}$. The associated Cayley graph is a homogeneous tree of degree 4. We denote its boundary by $\p \Fo_2$. By erasing the edges between all the consecutive powers of $a$ one obtains a forest $\Fc=\{\Tc_i\}_{i\in\ZZ}$ on $\Fo_2$, where $\Tc_i$ denotes the connected component of $a^i$, see \figref{fig:free}. Then the boundary $\p\Fc$ is obtained from $\p \Fo_2$ by removing the boundary points $\xi^\pm$ corresponding to the infinite words $a^\infty$ and $a^{-\infty}$, respectively, whereas its $\Fo_2$-invariant part $\p_{\Fo_2}\Fc$ is obtained from $\p \Fo_2$ by removing the whole $\Fo_2$-orbits of the points~$\xi^\pm$.
\end{ex}

\begin{figure}[h]
\begin{center}
\psfrag{a}[cl][cl]{$a^{-2}$}
\psfrag{b}[cl][cl]{$a^{-1}$}
\psfrag{c}[cl][cl]{$e$}
\psfrag{d}[cl][cl]{$a$}
\psfrag{e}[cl][cl]{$a^2$}
\psfrag{f}[cl][cl]{$a^\infty$}
\psfrag{g}[cl][cl]{$a^{-\infty}$}
\psfrag{p}[cl][cl]{$\Tc_{-2}$}
\psfrag{q}[cl][cl]{$\Tc_{-1}$}
\psfrag{r}[cl][cl]{$\Tc_0$}
\psfrag{s}[cl][cl]{$\Tc_1$}
\psfrag{t}[cl][cl]{$\Tc_2$}
\includegraphics[scale=.6]{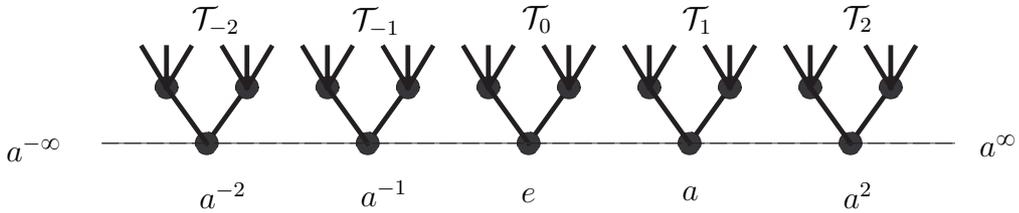}
\end{center}
\caption{A forest on the free group with two generators $\Fo_2$, for which the $\Fo_2$-invariant part of the boundary is not closed.}
\label{fig:free}
\end{figure}

\subsection{The boundary action} \label{sec:bdract}

Let now $\La=(\la,\Sib,\Ab)$ be a ladder on a group $G$. Since the associated forest $\Fc=\Fc(\La)$ is rooted at the unspiked vertices (see \prpref{prp:root}), we can identify the boundary $\p\Fc$ with the set of all geodesic rays $\gab=(\ga_0,\ga_1,\dots)$ issued from the roots $\ga_0\in\Uc$. These rays are characterized by the property that for any $n\ge 1$ the prefix $\sl\ga_n$ in the spike decomposition of the vertex $\ga_n$ is precisely the preceding vertex~$\ga_{n-1}$. In particular,
$$
\bn{ \ga_{n-1} } < \bn{ \ga_n } \qquad \text{and} \qquad
\left| \ga_{n-1} \right|_{\bn{ \ga_n }} \le \la (\bn{ \ga_n }) \qquad\forall\,n\ge 1 \;.
$$

\begin{prp} \label{prp:free}
If $\Fc=\Fc(\La)$ is the forest determined by a ladder $\La=(\la,\Sib,\Ab)$ on a group $G$, then the $G$-invariant part $\p_G\Fc$ of its boundary is non-empty, and the action of the group $G$ on $\p_G\Fc$ is free, i.e., $\Stab\xi=\{e\}$ for any $\xi\in\p_G\Fc$.
\end{prp}

\begin{proof}
Let us first notice that if $g = \sl g \cdot \wh g \cdot \sr g$ is the spike decomposition of $g\in G$, then $\wh g \cdot \sr g$ is the spike decomposition of $(\sl g)^{-1} g$, so that, in particular, $g$ and $(\sl g)^{-1} g$ have the same spike $\wh g$ and the same height $\bn g  =\bn{ (\sl g)^{-1} g }$. Thus, if $\gab=(\ga_0,\ga_1,\dots)$ is a geodesic ray issued from a root $\ga_0\in\Uc$, then
$$
\bn{ \ga_n } = \bn{ (\sl{\ga_n})^{-1} \ga_n } = \bn{ \ga_{n-1}^{-1}\ga_n } \qquad\forall\,n\ge 1\;.
$$

If $\gab$ represents a point $\xi\in\p_G\Fc$, then for any $g\in G$ the ray representing the translate~$g\xi$ eventually consists of the translates $g\ga_n$ with the same increments as the ray $\gab$ (see \remref{rem:bdry} and \figref{fig:action}), whence $\bn{ g\ga_n }=\bn{ \ga_n }$ for all sufficiently large $n$. Since the height $\bn\cdot$ is strictly increasing along any geodesic ray issued from $\Uc$, it implies that if $g\xi=\xi$, then eventually $g\ga_n=\ga_n$ (see \remref{rem:stab}), and therefore $g=e$.

For showing that the $G$-invariant part $\p_G\Fc$ of the boundary is non-empty we introduce a subset $\p_\sh\Fc\subset\p\Fc$ which consists of all boundary points $\xi\in\p\Fc$ for which the geodesic ray $\gab=(\ga_0,\ga_1,\dots)$ joining $\xi$ with its root $\ga_0$ has the property that
%$$
\begin{equation} \label{eq:minus}
\la\left(\bn{ \ga_n }\right) - \left| \ga_{n-1} \right|_{\bn{ \ga_n }}  \xrightarrow[n\to\infty]{} \infty \;.
\end{equation}
%$$

In the same way as in the proof of \prpref{prp:root}(iii) it is easy to see that for any $g\in G$ there is an integer $n$ such that $|g|_n \le \la(n)/2$. Then for any $\si\in\Si_n$ the product $g\myp =g\si$ has the property that $\bn{ g\myp }=n$, $\sl{g\myp}=g$, and
$$
\la(\bn{ g\myp }) - |g|_{\bn{ g\myp }} = \la(n) - |g|_n \ge \la(n)/2 \;.
$$
By starting from any root and iterating this procedure one obtains a geodesic ray converging to a point from $\p_\sh\Fc$, so that $\p_\sh\Fc$ is non-empty (cf. \secref{sec:Markov} above).

Finally, if $\xi\in\p_\sh\Fc$, i.e., if condition \eqref{eq:minus} is satisfied, then for any $g\in G$
one also has
$$
\la\left(\bn{ \ga_n }\right) - \left| g\ga_{n-1} \right|_{\bn{ \ga_n }} \ge
\la\left(\bn{ \ga_n }\right) - \left| \ga_{n-1} \right|_{\bn{ \ga_n }}
- \left| g \right|_{\bn{ \ga_n }}
\xrightarrow[n\to\infty]{} \infty \;,
$$
so that for all sufficiently large $n$
$$
g\ga_n = g \left( \sl{\ms\ga_n} \cdot \wh{\,\ms\ga_n\,} \cdot \sr{\ms\ga_n}\right)
= g \left( \ga_{n-1} \cdot \wh{\,\ms\ga_n\,} \cdot \sr{\ms\ga_n}\right)
= g \ga_{n-1} \cdot \wh{\,\ms\ga_n\,} \cdot \sr{\ms\ga_n} \;,
$$
is the spike decomposition of $g\ga_n$, and $\bn{ g\ga_n }=\bn{ \ga_n }$, whence $g\xi\in\p_\sh\Fc$. Thus, the set~$\p_\sh\Fc$ is $G$-invariant, and therefore it is contained in $\p_G\Fc$, so that
$\p_G\Fc\supset \p_\sh\Fc$ is also non-empty.
\end{proof}

\section{Records in sequences of i.i.d.\ variables} \label{sec:records}

\subsection{Simple records}

Let
$$
p = (p_0, p_1, \dots)
$$
be a \textsf{probability measure} on $\ZZ_+$, and let
$$
F_j = p_0 + p_1 +\dots + p_j
$$
be the corresponding \textsf{distribution function}. Given a sequence $X_1, X_2, \dots$ of \textsf{independent $p$-distributed random variables}, we denote by
%$$
\begin{equation} \label{eq:rec}
M_n = \max \{ X_1, X_2,\dots, X_n \}
\end{equation}
%$$
the \textsf{time $n$ record value} of the sequence, and by
$$
K_n = \card \bigl\{ i\in\{1,2,\dots,n\}: X_i=M_n \bigr\}
$$
its \textsf{multiplicity}\footnotemark, i.e., the number of times the record value $M_n$ has been attained on the interval between 1 and $n$. We say that the record value $M_n$ is \textsf{simple} (at time $n$) if $K_n=1$.

\footnotetext{\;One also refers to $K_n$ as the \textsf{number of variables tied for the record at time $n$} (by calling simple records \textsf{untied}). The somewhat surprising asymptotic behaviour of $K_n$ is an interesting probabilistic problem, see Eisenberg \cite{Eisenberg09} for a comprehensive overview.}

We are interested in the situation when
%$$
\begin{equation} \label{eq:kn}
K_n=1 \; \text{for all sufficiently large}\; n \;,
\end{equation}
%$$
i.e., when only a \emph{finite} number of records along the whole sequence $(X_n)$ are not simple. In this case we say that the records of the sequence $(X_n)$ are \textsf{eventually simple}.

\begin{lem}[\textsf{eventual record simplicity}] \label{lem:bsw}
The records of a sequence $(X_n)$ of independent $p$-distributed $\ZZ_+$-valued random variables are almost surely eventually simple if and only if the measure $p$ is infinitely supported and
%$$
\begin{equation} \label{eq:sum}
\sum_{j=0}^\infty \rho_j^2 < \infty \;,
\end{equation}
%$$
where
%$$
\begin{equation} \label{eq:rj}
\rho_j = \frac{p_j}{1-F_{j-1}} = \frac{p_j}{p_j + p_{j+1} + \dots} \;.
\end{equation}
%$$
\end{lem}

\begin{cor} \label{cor:br}
If the sequence of the weights $p_j$ is slowly varying in the sense that there is $\ep>0$ such that
$$
p_i \ge \ep p_j \qquad \textup{for all}\;\;i,j\in\NN\;\; \textup{with}\;\; j \le i \le 2j
$$
(in particular, if $p_j$ has polynomial decay), then the records are almost surely eventually simple.
\end{cor}

The sufficiency and the necessity of condition \eqref{eq:sum} are due to Brands -- Steutel -- Wilms \cite[Theorem~2.3]{Brands-Steutel-Wilms94} and to Qi \cite[Theorem 2]{Qi97}, respectively. Later a shorter and more conceptual proof of \lemref{lem:bsw} in both directions was given by Eisenberg \cite[Theorem 3 and Corollaries 3.1, 3.2]{Eisenberg09}. It is a combination of the observations that

\begin{enumerate}[\rm(i)]
\item
the events
%$$
\begin{equation} \label{eq:bj}
B_j = \bigl\{ j \;\, \text{is a non-simple record value of the sequence} \; (X_n) \bigr\}
\end{equation}
%$$
are jointly independent;
\item
the probability of $B_j$ is $\rho_j^2$;
\item
the records of the sequence $(X_n)$ are eventually simple if and only if only finitely many of $B_j$ occur;
\end{enumerate}

\noindent
with the classical Borel -- Cantelli Lemma applied to the sequence $(B_j)$. It is much easier, however, to make these observations by using the \emph{record stopping times} determined by the sequence $(X_n)$.

\subsection{Record times}

Let
\medskip
%$$
\begin{equation} \label{eq:tk}
\begin{aligned}
T_1 &= 1 \;, \\
T_{k+1} &= \min \{ n> T_k : X_n \ge X_{T_k} \} = \min \{ n>T_k : X_n=M_n \} \;,\qquad k\ge 1
\end{aligned}
\end{equation}
%$$
be the sequence of the \textsf{record times} of $(X_n)$, and let
$$
R_k=X_{T_k}
$$
be the associated sequence of the \textsf{record values}\footnotemark. The sequence $(R_k)$ is non-decreasing, but it may well contain repeated entries (corresponding to non-simple records), and
$$
K_n = \card \bigl\{ k\le n: R_k=M_n \bigr\} \;,
$$
so that condition \eqref{eq:kn} in these terms means that the sequence $(R_k)$ is eventually \emph{strictly} increasing, i.e., in the usual terminology, that eventually any weak record is an ordinary one.

\footnotetext{\;In the usual definition one requires that each new record \emph{strictly} exceed the previous one, see Arnold~-- Balakrishnan -- Nagaraja \cite{Arnold-Balakrishnan-Nagaraja98} and Nevzorov \cite{Nevzorov01} for a general overview. Our definition is due to Vervaat \cite{Vervaat73} who used the terms \emph{weak} record epochs and \emph{weak} record values, respectively. In the present paper we omit the qualifier ``weak'' as we are only dealing with the weak record times.}

Vervaat noticed that the sequence $(R_k)$ is a Markov chain on $\ZZ_+$ with the initial state $R_0=0$ and the transition probabilities
$$
p_{ij} =
\begin{cases}
\displaystyle \frac{p_j}{p_i+p_{i+1}+\dots} &\;, \; i\le j \\ \\
\hfil 0 \hfil &\;, \; \text{otherwise} \;.
\end{cases}
$$
Or, in terms of the coefficients $\rho_j$ \eqref{eq:rj}
%$$
\begin{equation} \label{eq:pij}
p_{ij} =
\begin{cases}
\hfil \rho_i \hfil &\;, \; i=j \;, \\ \\
(1-\rho_i) \dots (1-\rho_{j-1}) \rho_j &\;, \; i<j \;, \\ \\
\hfil 0 \hfil &\;, \; i>j \;.
\end{cases}
\end{equation}
%$$
This observation leads to a simple proof of the following property \cite[Lemma~4.3]{Vervaat73} (also see Stepanov \cite[Lemma 1]{Stepanov92}), which was apparently not known to the authors of \cite{Brands-Steutel-Wilms94, Qi97, Eisenberg09} (we present Vervaat's argument in a somewhat modified form).

\begin{lem}
Let $(X_n)$ be a sequence of independent $p$-distributed $\ZZ_+$-valued random variables, and let
$$
\begin{aligned}
Z_j &= \card \bigl\{ n\ge 1: M_n=X_n=j \bigr\} \\[.3cm]
&= \max \bigl\{ K_n: M_n=j \bigr\} \\[.3cm]
&= \card \bigl\{ k\ge 1: R_k=j \bigr\}\\
&
\end{aligned}
$$
be the number of times a given value $j\ge 0$ happens to be the record value along the whole sequence $(X_n)$, i.e., $Z_j$ are the \textsf{occupation times} of the Markov chain $(R_k)$. Then each $Z_j$ is distributed geometrically with the parameter $(1-\rho_j)$, and, moreover, the family of random variables $(Z_j)$ is independent.
\end{lem}

\begin{proof}
Since the transition probabilities $p_{ij}$ are non-zero only when $i\le j$, for any non-negative $z_0,z_1,\dots,z_{j-1}$ and any $z_j>0$
$$
\Pb \bigl( Z_0=z_0,Z_1=z_1,\dots,Z_j=z_j \bigr)
$$
is the probability that the chain $(R_k)$ consecutively visits each of the points $i=0,1,\dots,j$ the prescribed number of times $z_i$, after which it leaves the point $j$. By formula \eqref{eq:pij} the corresponding product of transition probabilities \eqref{eq:pij} is precisely
$$
\prod_{i=0}^j (1-\rho_i) \rho_i^{z_i} \;,
$$
which implies the claim.
\end{proof}

In order to obtain \lemref{lem:bsw} it now remains to notice that the sets $B_j$ \eqref{eq:bj} are precisely the events $(Z_j\ge 2)$, so that, indeed, $\Pb(B_j)=\rho_j^2$, and to apply the Borel -- Cantelli Lemma.

\medskip

Below we shall also need a modification of \lemref{lem:bsw} in the presence of additional $\{0,1\}$-valued random variables $\ep_n$ (which can be considered as the indicators of two subpopulations of the general population). More precisely, let
\begin{enumerate}[\ib]
\item
$\wt p=\left(\wt p_{j,\ep}\right)$ be a probability measure on $\ZZ_+\times\{0,1\}$,
\item
$p=(p_j)$ be its projection onto $\ZZ_+$,
\item
$(X_n,\ep_n)$ be a sequence of independent $\wt p$-distributed random variables,
\end{enumerate}
so that $X_n$ are independent $p$-distributed. We are interested in the situation when the record values of the sequence $X_n$ are eventually realized only on the subpopulation determined by the condition $\ep_n=0$.

\begin{lem} \label{lem:pos}
If a probability measure $\wt p$ on $\ZZ_+\times\{0,1\}$ is such that
\begin{enumerate}[\rm(i)]
\item
its projection $p$ onto $\ZZ_+$ satisfies the eventual record simplicity condition of \lemref{lem:bsw}
$$
\sum_j \left( \frac{p_j}{p_j + p_{j+1} + \dots} \right)^2<\infty \;,
$$
\item
the coefficients
$$
\al_j = \frac{\wt p_{j,1}}{p_j} = \frac{\wt p_{j,1}}{p_{j,0}+p_{j,1}}
$$
satisfy condition
$$
\sum_j \al_j<\infty \;.
$$
\end{enumerate}
Then almost surely $\ep_{T_k}=0$ for all sufficiently large $k$, where $T_k$ are the record times \eqref{eq:tk} of the sequence $(X_n)$.
\end{lem}

\begin{proof}
Conditioned by the sequence $(T_k)$ and the sequence of the corresponding record values $R_k=X_{T_k}$, the values $\ep_{T_k}$ are independently sampled with the probabilities
$$
\Pb(\ep_{T_k}=0)=1- \al_{R_k} \;, \qquad \Pb(\ep_{T_k}=1)=\al_{R_k} \;.
$$
Since by \lemref{lem:bsw} the sequence $(R_k)$ almost surely contains only finitely many repetitions, the series $\sum_k \al_{R_k}$ is convergent, so that the claim follows from the classical Borel -- Cantelli lemma.
\end{proof}

\subsection{Estimates of record values}

Finally, we shall also need the following simple lower and upper estimates of the record values of a sequence $(X_n)$ of i.i.d.\ $\ZZ_+$-valued random variables.

\begin{lem} \label{lem:ph}
For any distribution $p$ on $\ZZ_+$ with infinite support there exist non-de\-creas\-ing functions $\ph,\psi:\NN\to\NN$ such that almost surely the record values $M_n$ \eqref{eq:rec} of the sequence $(X_n)$ satisfy the inequalities
$$
\ph(n) \le M_n \le \psi(n)
$$
for all sufficiently large $n$.
\end{lem}

\begin{proof}
Since
$$
\Pb \bigl( M_n \le M \bigr) = \Pb \bigl( X_1, X_2,\dots, X_n \le M \bigr)  = \bigl[\Pb \bigl( X_1 \le M \bigr)\bigr]^n = F^n_M \;,
$$
for the lower bound it is sufficient to require (by the Borel -- Cantelli lemma) that
$$
\sum_n F^n_{\ph(n)} <\infty \;.
$$
As for the upper bound, eventually $M_n\le \psi(n)$ if and only if eventually $X_n \le \psi(n)$, which (again by Borel -- Cantelli) holds almost surely if
%$$
\begin{equation} \label{eq:psi}
\sum_n \left( 1- F_{\psi(n)} \right) < \infty \;.
\end{equation}
%$$
Thus, for instance, one can define $\ph$ and $\psi$ in such a way that $F_{\ph(n)}^n$ and $1- F_{\psi(n)}$ both behave like $1/n^2$, i.e.,
$$
F_{\ph(n)} \approx \left( \frac1{n^2} \right)^{1/n} = e^{-2\log n/n} \approx 1- 2\,\frac{\log n}n \;,
$$
and
$$
F_{\psi(n)} \approx 1 - \frac1{n^2} \;.
$$
\end{proof}

\begin{rem}
As it was proved by Barndorff--Nielsen \cite[Theorem 1]{BarndorffNielsen61}, under the additional assumption that $F^n_{\ph(n)}$ is non-increasing the condition
$$
\sum_{n=3}^\infty F_{\ph(n)}^n \, \frac{\log\log n}{n} < \infty
$$
is actually necessary and sufficient for the lower bound from \lemref{lem:ph}. As for the upper bound, since $X_n$ are independent, it is in fact equivalent to \eqref{eq:psi}, which goes back to Geffroy \cite{Geffroy58}. See Tomkins -- Wang \cite{Tomkins-Wang98} for a survey of the related results.
\end{rem}

Reformulated in terms of the record times $T_k$ \eqref{eq:tk} and the record values $R_k=X_{T_k}$, \lemref{lem:ph} then leads to

\begin{lem} \label{lem:psi}
For any probability distribution $p$ on $\ZZ_+$ with infinite support there exist non-decreasing functions $\Phi,\Psi:\NN\to\NN$ such that almost surely
$$
T_{k+1} \le \Phi(R_k) \;, \qquad R_{k+1} \le \Psi(R_k)
$$
for all sufficiently large $k$.
\end{lem}

\begin{proof}
Since $M_n=R_k$ on the whole interval $T_k\le n\le T_{k+1}-1$, by the lower bound from \lemref{lem:ph}
$$
R_k \ge \ph(T_{k+1}-1) \;,
$$
whence, by passing to the inverse function of $\ph$, one obtains the desired upper bound for the record time $T_{k+1}$. As for the record value $R_{k+1}$, by the upper bound from \lemref{lem:ph} and the just established inequality for $T_{k+1}$,
$$
R_{k+1} \le \psi (T_{k+1}) \le \psi (\Phi (R_k)) \;.
$$
\end{proof}

\section{Identification of the Poisson boundary} \label{sec:poisson}

\subsection{Random walks on groups} \label{sec:rw}

Let $G$ be a \textsf{countable group}, and $\mu$ be a \textsf{probability measure} on~$G$. We denote the \textsf{support of the measure} $\mu$ by
$$
\supp\mu = \{g\in G: \mu(g)>0 \} \;,
$$
and use the notations $\sgr\mu$ and $\gr\mu$ for the \textsf{subsemigroup} and the \textsf{subgroup of $G$ generated by} $\supp\mu$, respectively. The measure~$\mu$ is called \textsf{irreducible}\footnotemark\, (resp., \textsf{non-degenerate}) with respect to the group $G$ if
$$
\gr\mu=G \qquad (\text{resp.,}\; \sgr\mu=G) \;.
$$
Note that the irreducibility assumption is \emph{not restrictive} (unlike non-degeneracy), as any measure $\mu$ can be made irreducible by passing to the group $\gr\mu$, and from now on
$$
\textsf{all probability measures on groups are assumed to be irreducible\;.}
$$

\begin{rem}
If the measure $\mu$ is \textsf{symmetric}, i.e., $\mu(g)=\mu\left(g^{-1}\right)$ for any $g\in G$, then non-degeneracy and irreducibility of $\mu$ are obviously equivalent.
\end{rem}

\footnotetext{\;
Historically, the first term for the measures $\mu$ with $\gr\mu=G$ was \emph{aperiodic} used by Spitzer \cite[Section I.2, D2]{Spitzer64} who pointed up that if this condition is not satisfied, then ``the problem is badly posed. In other words, \dots the random walk is defined on the \emph{wrong group}''. Right away he also admitted that his ``terminology differs from the conventional one in the theory of Markov chains''. This was precisely the reason why Guivarc'h -- Keane -- Roynette \cite[D\'efinition 20]{Guivarch-Keane-Roynette77} replaced it with the term \emph{adapted (adapt\'ee)}, which is the currently prevalent usage. However, we feel that \emph{irreducible} is much more explicit and informative: any measure $\mu$ can be made irreducible by \emph{reducing} the group $G$ to the subgroup $\gr\mu$.}

The \textsf{random walk} $(G,\mu)$ is the Markov chain on $G$ whose \textsf{transition measures} are the translates $\pi_g=g\mu$ of the measure $\mu$ (which is often called the \textsf{step distribution} of the random walk). In other words, the \textsf{Markov transition}
%$$
\begin{equation} \label{eq:trans}
g \mapstoto^{h\sim\mu} gh
\end{equation}
%$$
from a point $g\in G$ consists in the (right) multiplication of $g$ by a random \textsf{increment} $h$ sampled from the distribution $\mu$. The \textsf{Markov operator} $P$ of the random walk $(G,\mu)$ acts on appropriate functions $f$ on $G$ by averaging as
%$$
\begin{equation} \label{eq:p}
Pf (g) = \sum_h \mu(h) f(gh) \;.
\end{equation}
%$$
The action of its \textsf{dual operator} on measures $\la$ on $G$ amounts to applying the Markov transitions \eqref{eq:trans}, i.e., to the (right) \emph{convolution} with $\mu$:
$$
\la P = \la*\mu \;.
$$

Given an \textsf{initial distribution} $\tha$ on $G$, we denote by $\Pb_\tha$ the associated \textsf{Markov measure} on the space $G^{\ZZ_+}$ of \textsf{sample paths} $\yb=(y_0,y_1,y_2,\dots)$, where
$$
y_n = y_0 x_1 x_2 \dots x_n \;,
$$
and $(x_n)$ is the sequence of the independent $\mu$-distributed \textsf{increments} of the random walk. Therefore, the one-dimensional distribution of the measure $\Pb_\tha$ at time $n$ is $\tha P^n=\tha*\mu^{*n}$, where $\mu^{*n}$ denotes the \textsf{$n$-fold convolution} of the measure $\mu$ with itself. For the initial distribution $\de_g$ concentrated at a single point $g\in G$ we use the notation $\Pb_g=\Pb_{\de_g}$, so that $\Pb_\tha=\sum_g \tha(g)\Pb_g$, and by $\Pb=\Pb_e$ we denote the measure on the path space issued from the identity $e$ of the group $G$. The path space $G^{\ZZ_+}$ is endowed with the (left) \textsf{coordinate-wise action} \eqref{eq:action} of the group $G$, so that $\Pb_g=g\Pb$, and the measure $\Pb_\tha$ is the convolution $\tha*\Pb$ of the initial distribution $\tha$ with the measure $\Pb$ with respect to this action.

We denote by \;$\N$ the \textsf{counting measure} on $G$. The associated $\si$-finite measure $\Pb_\N$ on the path space dominates $\Pb_\tha$ for any initial distribution $\tha$ and is equivalent to it if $\supp\tha=G$.

\subsection{The Poisson boundary}

The asymptotic equivalence relation $\sim$ \eqref{eq:eq} on the space $G^{\ZZ_+}$ of sample paths $\yb=(y_0,y_1,y_2,\dots)$ can also be described as the \textsf{orbit equivalence relation}
$$
\yb\sim\yb' \iff \exists\, n,n'\in\ZZ_+: T^n\yb=T^{n'}\yb' \;,
$$
of the \textsf{time shift}
%$$
\begin{equation} \label{eq:shift}
T: (y_0,y_1,y_2,\dots) \mapsto (y_1,y_2,y_3,\dots) \;.
\end{equation}
%$$

We denote by $\Eg$ the complete $\si$-algebra of measurable ($\Pb_\N$ -- mod 0) $T$-invariant subsets of $G^{\ZZ_+}$, i.e., the $\si$-algebra of all measurable subsets of the path space $(G^{\ZZ_+},\Pb_\N)$ which are unions ($\Pb_\N$ -- mod 0) of the equivalence classes of the relation $\sim$, and call it the \textsf{exit $\si$-algebra}\footnotemark\, of the random walk $(G,\mu)$. The quotient of the path space determined by the $\si$-algebra $\Eg$, i.e., the space of ergodic components of the shift $T$ with respect to the invariant measure $\Pb_\N$, is called the \textsf{Poisson boundary} $\Ps$ of the random walk $(G,\mu)$, and we shall denote by
$$
\bnd: G^{\ZZ_+} \to \Ps
$$
the associated quotient map. By
$$
\nu_\tha=\bnd\Pb_\tha
$$
we denote the \textsf{harmonic measure} determined by an initial distribution $\tha$ (which is well-defined because the associated measure $\Pb_\tha$ on the path is absolutely continuous with respect to $\Pb_\N$), and by $\nnu$ we denote the quotient \textsf{harmonic measure class} on $\Ps$, i.e., the common measure class of the harmonic measures $\nu_\tha$ corresponding to the initial distributions $\tha$ with $\supp\tha=G$.

\footnotetext{
\;The isomorphism of the Banach spaces of bounded harmonic functions on the state space and of bounded $\Eg$-measurable functions on the path space of an arbitrary countable Markov chain was established already by Blackwell \cite{Blackwell55} as a simple consequence of the martingale convergence theorem, see formula~\eqref{eq:limit} below. He called the $\si$-algebra $\Eg$ \emph{invariant}, and it was also called \emph{stationary} by Neveu~\cite{Neveu64}, but these words are overloaded, so that we prefer to borrow instead the term \emph{exit} from Doob \cite{Doob59} and Hunt \cite{Hunt60} who used it when talking about the \emph{exit boundaries} and \emph{exits} of a Markov chain.}

Since the action \eqref{eq:action} of the group $G$ on the path space $G^{\ZZ_+}$ commutes with the time shift \eqref{eq:shift}, it descends to the Poisson boundary, and the measure class $\nnu$ is quasi-invariant with respect to this action. Let
$$
\nu=\bnd\Pb
$$
denote the \textsf{harmonic measure corresponding to the group identity}. Then $\nu_g=g\nu$ for an arbitrary starting point $g\in G$, and $\nu_\tha=\tha*\nu$ for an arbitrary initial distribution $\tha$.

A function $f$ on the group $G$ is called \textsf{$\mu$-harmonic} if $Pf=f$ for the Markov operator~$P$ \eqref{eq:p}. By $H^\infty(G,\mu)$ we denote the Banach space of bounded $\mu$-harmonic functions on~$G$ endowed with the $\sup$-norm, on which the group $G$ acts by (left) translations. The harmonic measure $\nu$ is \textsf{$\mu$-stationary} in the sense that
$$
\nu = \bnd \Pb = \bnd (T\Pb) = \bnd \Pb_\mu = \sum_g \mu(g) g\nu = \mu*\nu \;.
$$
Therefore, for any $\wh f\in L^\infty(\Ps,\nnu)$ the \textsf{Poisson formula}
%$$
\begin{equation} \label{eq:poisson}
f(g)=\langle \wh f,g\nu\rangle
\end{equation}
%$$
produces a $\mu$-harmonic function. In fact, the correspondence \eqref{eq:poisson} is an isomorphism of the Banach spaces $H^\infty(G,\mu)$ and $L^\infty(\Ps,\nnu)$. Indeed, harmonicity of the function $f$ means that the sequence of functions
%$$
\begin{equation} \label{eq:limit}
F_n(\yb)=f(y_n)
\end{equation}
%$$
on the path space is a \emph{martingale} with respect to the increasing filtration of the coordinate $\si$-algebras $\Ag_0^n$. Therefore, by the martingale convergence theorem, the sequence $F_n$ converges almost surely to a $\Eg$-measurable limit function $F_\infty$, which produces a boundary function $\wh f$. Conversely, the Poisson integrals \eqref{eq:poisson} are essentially nothing other than the conditional expectations of the boundary function $\wh f$ with respect to the coordinate $\si$-algebras $\Ag_0^n$.

\subsection{The trunk criterion} \label{sec:trunk}

If the state space of a Markov chain is endowed with additional (geometrical, combinatorial, algebraic, etc.) structures, and the transition probabilities comply with (or, are governed by) them, then a natural question is to identify (describe) the Poisson boundary in terms of these structures. This problem usually splits into two quite different parts:
\begin{enumerate}[{\rm (i)}]
\item
to exhibit a (complete) sub-$\si$-algebra $\Bg$ of the exit $\si$-algebra $\Eg$, or, in terms of the associated quotient spaces, to exhibit a quotient $B$ of the Poisson boundary;
\item
to prove that the $\si$-algebra $\Bg$ coincides with the whole exit $\si$-algebra $\Eg$, i.e., that the ``candidate'' space $B$ is the Poisson boundary.
\end{enumerate}
In other words, first one has to exhibit a certain system of invariants (“patterns”) of the behaviour of the Markov chain at infinity, and then one has to show the completeness of this system, i.e., that these patterns completely describe the behaviour at infinity, see \cite{Kaimanovich96, Kaimanovich00a}. A particular case of this problem is proving that the Poisson boundary is trivial, in which situation the candidate space $B$ is just a singleton.

In the context of random walks on groups one usually applies the conditional entropy criterion \cite{Kaimanovich85, Kaimanovich00a} (a generalization of the entropy criterion for the triviality of the Poisson boundary \cite{Avez74, Derriennic80, Vershik-Kaimanovich79, Kaimanovich-Vershik83}) which requires finiteness of the entropy of the step distribution. However, there are situations when one can use the following very simple argument which imposes no \emph{a priori} assumptions on the considered Markov chains.

\begin{dfn} \label{dfn:trunk}
A \textsf{trunk}\footnotemark\, of a Markov chain on a countable state space $X$ is an infinite subset $A\subset X$ with the property that almost every sample path of the chain issued from any initial state visits all but finitely many points of $A$. We shall say that a Markov chain has the \textsf{trunk convergence} property with respect to a quotient $B$ of the Poisson boundary if for almost every point $b\in B$ the associated conditional Markov chain has a trunk.
\end{dfn}

\footnotetext{\;Our choice of terminology evokes the meaning of ``trunk'' as \emph{the main part of something as distinguished from its appendages} given by the Oxford English Dictionary, in particular,
\emph{the trunk of a tree.}}

\begin{prp}[\textsf{trunk criterion}] \label{prp:trunk}
If a Markov chain has the trunk convergence property with respect to a quotient $B$ of the Poisson boundary, then $B$ is actually the Poisson boundary of the chain.
\end{prp}

\begin{proof}
Let is consider first the situation when $B$ is a singleton, i.e., the original Markov chain itself has a trunk. Any bounded harmonic function $f$ converges along almost every sample path of the chain. However, by \dfnref{dfn:trunk} the limit boundary values \eqref{eq:limit} of $f$ are almost surely the same, so that $f$ is also constant, and therefore the Poisson boundary of the chain is trivial, i.e., coincides with the singleton~$B$.

Now, in the general case, if the conditions of the Proposition are satisfied, then by the above the Poisson boundary of almost every $B$-conditioned chain is trivial, which means that the projection of the Poisson boundary onto $B$ has trivial fibres.
\end{proof}

\begin{rem}
In what concerns transient random walks on trees, the idea that in the nearest neighbour case almost every sample path visits all points on the geodesic ray joining its starting and limit points was used already in the works of Dynkin~-- Malyutov \cite{Dynkin-Malutov61} and Cartier \cite{Cartier72} for the identification of the Martin boundary with the geometric boundary for the \emph{nearest neighbour random walks on free groups and trees}, respectively (also see Woess \cite{Woess86} for the case of the nearest neighbour random walks on general free products and a comprehensive exposition in Woess' book \cite[Section~26]{Woess00}). Our \prpref{prp:trunk} can be considered as an implementation of this idea in the conceptually simpler case of the Poisson boundary (cf.\ the comparative discussion of the Poisson and the Martin boundaries in \cite{Kaimanovich96}).

Yet another example of an application of the trunk criterion is provided by the \emph{occupation Markov chains in the presence of infinitely many strong cutpoints}. Given a Markov chain $(\xi_n)$ on a countable state space $X$, let the \textsf{occupation time}
$$
N_t(x)=\card\{n: 0\le n \le t,\; \xi_n=x\}
$$
be the number of visits of the chain to a point $x\in X$ up to time $t$. Then $(\xi_n,N_{n-1})$ is
the associated \textsf{occupation Markov chain}. If $(\xi_n)$ is the random walk on a group $G$ with a step distribution $\mu$, then the occupation chain is the random walk on the \emph{wreath product} $G\wr\ZZ$ with the step distribution $\mu\otimes\de_\uptheta$, where $\uptheta$ denotes the delta configuration at the identity of $G$ (see \cite{Kaimanovich85a, Kaimanovich91}). A point $\xi_k$ is called a \textsf{strong cutpoint} for a sample path $(\xi_n)$ if for any $0\le i < k <j$ the probability of transition from $\xi_i$ to $\xi_j$ is 0 (we use the terminology from James -- Lyons -- Peres \cite{James-Lyons-Peres08} which slightly differs from the one originally used by James -- Peres \cite{James-Peres96} who introduced this notion). As it was proved by James -- Peres \cite[Proposition 1.1]{James-Peres96}, if almost every sample path $(\xi_n)$ has infinitely many strong cutpoints, then the Poisson boundary of the occupation chain $(\xi_n,N_{n-1})$ coincides with the space of the limit occupation functions $N_\infty=\lim_n N_n$. This condition is satisfied for all transient simple random walks on groups, but there are examples of transient reversible Markov chains that almost surely have only a finite number of cutpoints (see \cite{James-Lyons-Peres08} and the references therein). Now, an infinite sequence of strong cutpoints is precisely a trunk for the occupation chain conditioned by the limit occupation function $N_\infty$.

Note that in all these situations the trunk condition is satisfied in a somewhat stronger form than required in \dfnref{dfn:trunk}: the trunk sets are actually linearly ordered which allows one to introduce a tree structure on the state space. Yet another (weaker) modification of the trunk condition can be obtained by requiring just that the intersection of the sets visited by any two sample paths be almost surely infinite; this is still sufficient to imply the result of \prpref{prp:trunk}.
\end{rem}

\subsection{Ladder adapted random walks}

We shall now combine the considerations from \secref{sec:ladders} concerning ladders and the associated forests on groups with the results on the records in i.i.d.\ sequences from \secref{sec:records} in order to obtain, by using the trunk criterion from \prpref{prp:trunk}, a complete description of the Poisson boundary of a class of random walks on groups admitting ladder structures.

\vfill\eject

\begin{thm} \label{thm:poi}
Let
\begin{enumerate}[\ib]
\item
$p=(p_j)$ be a probability measure on $\ZZ_+$ which satisfies the eventual record simplicity condition
of \lemref{lem:bsw}
$$
\sum_j \left( \frac{p_j}{p_j + p_{j+1} + \dots} \right)^2<\infty
$$
(for instance, any measure with polynomial decay);
\medskip
\item
$\Phi,\Psi:\NN\to\NN$ be the functions from \lemref{lem:psi} associated with the distribution $p$;
\medskip
\item
$\La=(\Sib,\Ab)$ be a $\Phi$-ladder on a countable group $G$ (see \dfnref{dfn:lad});
\medskip
\item
$\mu$ be a probability measure on the group $G$ with
$$
\supp\mu=\bigcup_{i=1}^\infty\Si_i \cup \bigcup_{i=0}^\infty A_i
$$
such that
$$
\mu(\Si_i\cup A_i)=p_i \;, \qquad \sum_i \frac{\mu(A_i)}{p_i}<\infty \;.
$$
\end{enumerate}
Then
\begin{enumerate}[{\rm (i)}]
\item
for any starting point $g\in G$ the sample paths of the random walk $(G,\mu)$ almost surely converge to the boundary $\p\Fc$ of the ladder forest $\Fc=\Fc(\La)$ (see \dfnref{dfn:forest}), and almost every sample path visits all sufficiently remote points on the corresponding limit geodesic ray in $\Fc$;
\item
the resulting harmonic (hitting) measures $\nu_g$ are concentrated on the $G$-invariant part $\p_G\Fc$ of the boundary $\p\Fc$;
\item
the Poisson boundary $\p_\mu G$ of the random walk $(G,\mu)$ coincides with the space $\p\Fc$ endowed with the family of measures $\nu_g$;
\item
the action of the group $G$ on the Poisson boundary $\p_\mu G$ is free (mod 0) with respect to the harmonic measure class, in particular, the Poisson boundary is non-trivial;
\item
the properties stated in \textup{(i)} also hold for the constrained subforest $\Fc_\Psi=\Fc_\Psi(\La)\subset \Fc$ determined by the ladder $\La$ and the constraining function $\Psi$ (see \dfnref{dfn:constr}), so that the hitting measures $\nu_g$ are in fact concentrated on the boundary $\p\Fc_\Psi\subset\p\Fc$.
\end{enumerate}
\end{thm}

\begin{proof}[Proof of \thmref{thm:poi}]

\;(i)\; We consider first the situation when the starting point $g$ is the group identity~$e$. Let $(x_n)$ be a sequence of independent $\mu$-distributed increments of the random walk $(G,\mu)$, and let
$$
y_n = x_1 x_2 \dots x_n
$$
be the associated sample path of the random walk. Further, let $T_k$ and $R_k=X_{T_k}$ be, respectively, the record times \eqref{eq:tk} and the associated record values of the sequence of independent $p$-distributed random variables $X_n=\ze(x_n)$, where $\ze(g)=j$ whenever $g\in \Si_j\cup A_j$. Then by \lemref{lem:bsw}, \lemref{lem:pos}, and \lemref{lem:psi}
almost surely there exists an integer $k_0$ such that for all $k\ge k_0$
\begin{enumerate}[{\rm (a)}]
\item
$R_k$ is simple, and therefore $X_n<R_k$ not only for all $n<T_k$, but also for all $n$ in the interval from $T_k$ to the next record time $T_{k+1}$;
\item
$x_{T_k}\in \Si_{R_k}$;
\item
$T_{k+1}\le \Phi(R_k)$.
\end{enumerate}
In terms of the increments $x_n$ (a) means that
%$$
\begin{equation} \label{eq:xn}
x_n \in \bigcup_{i=1}^{R_k-1} \Si_i \; \cup \; \bigcup_{i=0}^{R_k-1} A_i \subset \De_{R_k}
\qquad\forall\, n<T_{k+1}, n\neq T_k
\end{equation}
%$$
(see \eqref{eq:de} for the definition of the sets $\De_n$). Therefore, in view of the bound (c), for any $T_k\le n<T_{k+1}$ the position $y_n$ of the random walk at time $n$ is spiked with respect to the ladder $\La=(\Phi,\Sib,\Ab)$ (see \dfnref{dfn:spike}) with the spike
%$$
\begin{equation} \label{eq:yn}
\wh{\,y_n} = x_{T_k} \;,
\end{equation}
%$$
the height $\bn{ y_n }=R_k$, the prefix
$$
\sl{\ms y_n} = x_1 x_2 \dots x_{T_k-1} = y_{T_k-1} \;.
$$
and the postfix
$$
\sr{\ms y_n} = x_{T_k+1} \dots x_n = y_{T_k}^{-1} y_n \;.
$$
In particular,
$$
\sl{\ms y_{T_{k+1}-1}} = y_{T_k-1} \;.
$$
Thus, $\left(y_{T_k-1}\right)_{k\ge k_0}$ is a geodesic ray and the sequence~$(y_n)$ converges to a boundary point $y_\infty\in\p\Fc$, see \figref{fig:forestwalk} from the Introduction. Moreover, the ray $\left(y_{T_k-1}\right)_{k\ge k_0}$ is a trunk of the sample path $(y_n)$.

The above argument is also applicable to the situation when the starting point $g\in G$ is arbitrary. Indeed, its lengths \eqref{eq:length} eventually stabilize, see \eqref{eq:ni}. On the other hand, since $\mu(e)>0$ by assumption \eqref{eq:a0e}, the number of $e$'s among the first $n$ increments $x_1,x_2,\dots, x_n$ of the random walk almost surely goes to infinity as $n\to\infty$, so that for all sufficiently large $n$ the position $y_n$ is still spiked with the spike \eqref{eq:yn} and the prefix
$$
\sl{\ms y_n} = g x_1 x_2 \dots x_{T_k-1} = y_{T_k-1} \;.
$$

\;(ii)\; By assumption \eqref{eq:a0e}, $\mu(e)>0$, which implies that there are infinitely many identity increments in the sequence $(x_n)$, so that the limit of almost every sample path belongs to $\p_\sh\Fc\subset \p_G\Fc$ (see the proof of \prpref{prp:free}).

\;(iii)\; As we have already noticed in the proof of (i), the geodesic ray in $\Fc$ determined by a boundary point from $\p\Fc$ serves as a trunk for the $\p\Fc$-conditioned random walk. Therefore, the claim follows from the trunk criterion \prpref{prp:trunk}.

\;(iv)\; This is a consequence of \prpref{prp:free}.

\;(v)\; By \lemref{lem:psi} almost surely $R_{k+1} \le \Psi(R_k)$ for all sufficiently large $k$, which, in view of \eqref{eq:xn}, means that for all sufficiently large $n$ the edges $\left(y_n, \sl{\ms y_n}\right)$ are also present in the constrained subforest $\Fc_\Psi\subset \Fc$. Thus, almost every sample path $(y_n)$ also converges in $\Fc_\Psi$, and for sufficiently large $k$ the points $y_{T_k-1}$ form a geodesic ray in $\Fc_\Psi$.
\end{proof}

\section{Existence of ladders} \label{sec:sicc}

If a group $G$ has a non-trivial finite \textsf{conjugacy class}
%$$
\begin{equation} \label{eq:conj}
\Cc_g=\{h^{-1}gh : h\in G\} \;,
\end{equation}
%$$
then it has no ladders in the sense of \dfnref{dfn:lad}. Indeed, if $\La=(\la,\Sib,\Ab)$ is a ladder, and
$\Cc_g$ is finite, then $\Cc_g\subset\De_n^{\la(n)}$ for all sufficiently large $n$ by condition \eqref{eq:asi}, so that the identity
$$
g \cdot \si = \si \cdot ( \si^{-1} g \si)
$$
with $\si\in\Si_n$ produces a contradiction to the uniqueness of the spike decomposition. The purpose of this Section is to show that, in fact, the presence of non-trivial finite conjugacy classes is the only obstacle to the existence of ladders, and that ladders with arbitrary gauge functions exist on any \textsf{ICC group} (i.e., one in which the conjugacy class of any non-identity element is infinite).

\subsection{Switcher ladders}

We begin by describing a special class of ladders. Given a subset~$Z$ of a group $G$, we denote by $\mr{Z}=Z\sm\{e\}$ its \textsf{puncture} obtained by removing the group identity $e$.

\begin{dfn} \label{dfn:sset}
A non-empty set $\Si\subset \mr{G}$ is called \textsf{switching} for a subset $Z$ of a group $G$ (or \textsf{$Z$-switching} in short) if the equality
$$
\si z = z' \si'
$$
with $z,z'\in Z$ and $\si,\si'\in\Si$ is only possible when $z=z'=e$ (and therefore $\si=\si'$), in other words, if
$$
\Si \mr{Z} \cap Z \Si = \Si Z \cap \mr{Z} \Si = \vn \;.
$$
We denote by $\Sfb(Z)$ the collection of all $Z$-switching subsets of $G$.
\end{dfn}

\begin{rem} \label{rem:empty}
If $\Si$ is a $Z$-switching set, then obviously $\Si\cap Z=\vn$. In the case when $Z$ does not contain the group identity $e$, i.e., when $\mr{Z}=Z$, \dfnref{dfn:sset} just means that $\Si Z \cap Z \Si=\vn$. However, note the asymmetry of this definition with respect to $\Si$ and $Z$ (it will be used below): the group identity is allowed to belong to $Z$, but not to $\Si$.
\end{rem}

\begin{ex}
Let $\Fo_2=\langle a,b \rangle$ be the free group with free generators $a,b$. Then the set $a\Fo_2^+$ of all positive words which begin with $a$ is switching for the set $b\Fo_2^+$ of all positive words which begin with $b$. For yet another example of a similar nature let $\displaystyle G=\free_{i=1}^\infty G_i$ be the countable free product of countable groups $G_i$. Then for any $n\ge 1$
$$
\mr{\overgroup{{\left(\free_{i=n+1}^\infty G_i\right) }}} \in \Sfb \left( \free_{i=1}^n G_i \right) \;.
$$
\end{ex}

\begin{lem} \label{lem:5la}
If $\La=(\la,\Sib,\Ab)$ is a scale on a countable group $G$ such that
%$$
\begin{equation} \label{eq:5la}
\Si_n \in \Sfb \Bigl( \De_n^{5\la(n)} \Bigr) \qquad\forall\, n\ge 1 \;,
\end{equation}
%$$
then $\La$ is a ladder.
\end{lem}

\begin{proof}
Let us first notice that condition \eqref{eq:5la} implies, in view of \remref{rem:empty}, that
%$$
\begin{equation} \label{eq:side}
\Si_n \cap \De_n^{5\la(n)} = \vn \qquad\forall\, n\ge 1 \;,
\end{equation}
%$$
or, in other words, that
$$
|\si|_n > 5\la(n) \qquad\forall\,\si\in \Si_n,\; n \ge 1 \;.
$$
Thus, in view of \remref{rem:escape}, we only have to verify condition (i) from \dfnref{dfn:lad}. Let
%$$
\begin{equation} \label{eq:twospike}
g = g_- \cdot \si \cdot g_+ = g_-' \cdot \si' \cdot g_+'
\end{equation}
%$$
be two spike decompositions of the same $g\in G$ of heights $n$ and $n'$, respectively.

If $n\neq n'$, for instance, $n>n'$, then
$$
\si = g_-^{-1} \cdot g_-' \cdot  \si' \cdot g_+' \cdot  g_+^{-1} \in \De_n^{5\la(n)} \;,
$$
which contradicts \eqref{eq:side}.

If $n=n'$, then
$$
\si \cdot g_+ \cdot (g_+')^{-1} = g_-^{-1} \cdot g_-' \cdot \si'
$$
with $\si,\si'\in\Si_n$ and
$$
g_+ \cdot (g_+')^{-1}, g_-^{-1} \cdot g_-' \in \De_n^{2\la(n)} \subset \De_n^{5\la(n)} \;,
$$
which, in view of condition \eqref{eq:5la}, is only possible when the spike decompositions \eqref{eq:twospike} coincide.
\end{proof}

\subsection{Negligible subsets in groups} \label{sec:negl}

We shall need a notion of ``smallness'' in infinite groups.

\begin{dfn} \label{dfn:negl}
A subset of an infinite group $G$ is called \textsf{negligible} if it can be covered by a finite union of cosets of infinite index subgroups of $G$, and it is called \textsf{non-negligible} otherwise.
\end{dfn}

\begin{rem}
Any right coset $Hg$ can be presented as a left coset $g\cdot g^{-1} H g$ of the conjugate subgroup $g^{-1} H g$, so that the class of negligible sets introduced in \dfnref{dfn:negl} remains the same if one considers the left cosets only.
\end{rem}

\begin{rem} \label{rem:dens}
Any left-invariant mean $m$ on an amenable group $G$ assigns measure 0 to any coset of any infinite index subgroup. Therefore, in this situation $m(A)=0$ for any negligible subset $A\subset G$, but the converse is far from being true. For instance, the only infinite index subgroup of the group of integers $\ZZ$ is the identity subgroup, so that a subset $A\subset\ZZ$ is negligible if and only if it is finite, and therefore the set $A=\{2^n: n\ge 0\}\subset\ZZ$ is non-negligible. However, its measure is 0 for any invariant mean on $\ZZ$, because the intersection of any two translates of $A$ contains at most one point.
\end{rem}

Since any group element can be considered as a coset of the identity subgroup, our starting observation is obvious:

\begin{lem} \label{lem:ab}
If $A\subset G$ is non-negligible, and $B\subset G$ is negligible, then $A\sm B$ is infinite.
\end{lem}

The \emph{raison d'\^etre} of \dfnref{dfn:negl} is

\begin{lem}[\textsf{Neumann's Lemma}] \label{lem:negl}
Any infinite group is non-negligible, i.e., it can not be covered by a finite union of cosets of infinite index subgroups.
\end{lem}

Then \lemref{lem:ab} and \lemref{lem:negl} imply

\begin{cor}\label{cor:negl}
If the complement $G\sm A$ of a subset $A$ of an infinite group $G$ is negligible, then $A$ is infinite.
\end{cor}

The proof below is a somewhat streamlined version of the argument used by Neumann in the course of establishing a \emph{quantitative version} of \lemref{lem:negl}: if a group is covered by a finite union of cosets, then the sum of the inverse indices of the corresponding subgroups is not smaller than 1 \cite[(4.5)]{Neumann54}.

\begin{proof}[Proof of \lemref{lem:negl}]
We have to show that an infinite group $G$ can not be presented as
%$$
\begin{equation} \label{eq:cd}
G = \bigcup_{i=1}^n A_i G_i \;,
\end{equation}
%$$
where the subsets $A_i\subset G$ are finite, and the infinite index subgroups $G_i$ are now assumed to be \emph{distinct}. For $n=1$ this is clearly impossible, because the index of $G_1$ is infinite.

Therefore, if one has a presentation \eqref{eq:cd} with $n\ge 2$, then $G\neq A_1 G_1$, which means that there exists $g\in G\sm A_1 G_1$, or, equivalently, such that $g G_1\cap A_1 G_1=\vn$. Then \eqref{eq:cd} implies that
$$
gG_1 \subset \bigcup_{i=2}^n A_i G_i \;,
$$
whence
$$
A_1 G_1 \subset \bigcup_{i=2}^n A_1 g^{-1} A_i G_i \;,
$$
and one can rewrite \eqref{eq:cd} by using one subgroup less. By continuing this descent one arrives at $n=1$, which is impossible as we have already seen.
\end{proof}

\subsection{Switching and superswitching elements} \label{sec:s}

The switching elements defined by Frisch -- Tamuz -- Vahidi Ferdowsi in \cite{Frisch-Tamuz-Ferdowsi18} are, in our terminology, the one-point switching sets (see \dfnref{dfn:sset}).

\begin{dfn} \label{dfn:s}
Given a group $G$ and a subset $Z\subset G$, an element $\si\in G$ is called \textsf{$Z$-switching} if the one-point set $\Si=\{\si\}$ is $Z$-switching, i.e., if the equality
$$
\si x \si^{-1} = y
$$
with $x,y\in Z$ is only possible if $x=y=e$. We denote by $\Sf_1(Z)
\subset G$ the set of all $Z$-switching elements.
\end{dfn}

For any $x,y\in G$ let
%$$
\begin{equation} \label{eq:sxy}
\Sc_{x\to y} = \{ g\in G : g x g^{-1} = y \}
\end{equation}
%$$
be the set of elements $g$ whose action on $G$ by conjugation moves $x$ to $y$, so that the stabilizer
$$
\Sc_{x\to x} = \{ g\in G : g x g^{-1} = x \} = \{ g\in G : g x = x g \} = \C(x)
$$
is the centralizer of $x$ in $G$, and
%$$
\begin{equation} \label{eq:sgxx}
\Sc_{x\to y} = g \Sc_{x\to x} = g \C(x) \qquad\forall\,g\in \Sc_{x\to y} \;.
\end{equation}
%$$
In these terms
%$$
\begin{equation} \label{eq:us}
\Sf_1(Z) = G \;\; \sm \bigcup_{x,y\in Z\sm\{e\}} \Sc_{x\to y} \;.
\end{equation}
%$$

If $G$ is an ICC group, then the index of $\C(x)$ in $G$ is infinite for any non-trivial element $x\in G$ (see \secref{sec:icc} below), so that in this situation $\Sf(Z)$ is the complement of a negligible set by \eqref{eq:sgxx} and \eqref{eq:us}, and \corref{cor:negl} implies

\begin{lem}[{\cite[Claim 3.3]{Frisch-Tamuz-Ferdowsi18}}] \label{lem:s}
For any finite subset $Z$ of an ICC group the associated set $\Sf_1(Z)$ of switching elements is infinite.
\end{lem}

The super-switching elements defined by Frisch -- Hartmann -- Tamuz -- Vahidi Ferdowsi \cite{Frisch-Hartman-Tamuz-Ferdowsi18} correspond, in our terminology, to the switching sets of the form $\{\si,\si^{-1}\}$, see \dfnref{dfn:sset}.

\begin{dfn} \label{dfn:ss}
Given a group $G$ and a subset $Z\subset G$, an element $\si\in G$ is called \textsf{$Z$-superswitching} if the set $\Si=\{\si,\si^{-1}\}$ is $Z$-switching, i.e., if the equality
$$
\si x \si^\ep = y
$$
with $x,y\in Z$ and $\ep=\pm 1$ is only possible if $x=y=e$ (so that either $\ep=-1$, or $\ep=1$ and $\si$~is an involution). We denote by $\Sf_{\pm 1}(Z)\subset G$ the \textsf{set of all $Z$-superswitching elements}.
\end{dfn}

In the same way as we have defined the sets $\Sc_{x\to y}$ \eqref{eq:sxy}, let us introduce the sets
$$
\breve\Sc_{x\to y} = \{g\in G: gxg=y \} \;, \qquad x,y\in G \;,
$$
so that in these terms
%$$
\begin{equation} \label{eq:u}
\Sf_{\pm 1}(Z) = G \sm \left[ \; \bigcup_{x,y\in Z\sm\{e\}} \Sc_{x\to y} \quad
\cup \; \bigcup_{x,y\in Z: (x,y)\neq (e,e)} \breve\Sc_{x\to y} \;
\right]
\end{equation}
%$$
(the difference in the parameterizations of the unions of the sets $\Sc_{x\to y}$ and $\breve\Sc_{x\to y}$ is due to the fact that $\Sc_{x\to e}$ and $\Sc_{e\to x}$ are trivially empty for $x\neq e$). One can still link the sets $\breve\Sc_{x\to y}$ with the cosets of centralizers (cf.\ \eqref{eq:sgxx} above), although in a somewhat more complicated way.

\begin{lem} \label{lem:sixy}
$\breve\Sc_{y\to x} \, \breve\Sc_{x\to y}\subset \C(xy)$ for any $x,y\in G$, i.e., equivalently, $\breve\Sc_{x\to y}\subset g\C(xy)$ for any $x,y\in G$ and $g\in\breve\Sc_{x\to y}$.
\end{lem}

\begin{proof}
Let $g\in\breve\Sc_{y\to x}=\breve\Sc_{x\to y}^{-1}$ and $h\in\breve\Sc_{x\to y}$, i.e., $gyg=x$ and $hxh=y$, so that one has the ``commutation relations''
$$
hx = y h^{-1} \;,\quad gy = x g^{-1} \;,\quad h^{-1} y = xh \;,\quad g^{-1}x = yg \;.
$$
Their consecutive application (at the underlined places) allows one to ``move'' $x$ and $y$ past $gh$ and to obtain the identity
$$
g\underline{hx\ms}y
= \underline{gy\ms} h^{-1} y
= x g^{-1} \underline{h^{-1} y\ms}
= x \underline{g^{-1} x\ms} h = xy gh \;.
$$
\end{proof}

In order to apply the idea of the proof of \lemref{lem:s} to superswitching elements it remains to consider the sets $\breve\Sc_{x\to x^{-1}}$. We denote by
$$
\Ic = \{g\in G: g^2=e \}
$$
the \textsf{set of involutions in $G$}. Then
%$$
\begin{equation} \label{eq:i}
\breve\Sc_{x\to x^{-1}} = x^{-1}\Ic \qquad \forall\,x\in G \;.
\end{equation}
%$$

Now we are ready to prove

\begin{prp} \label{prp:ss}
For any finite subset $Z$ of an ICC group the associated set $\Sf_{\pm 1}(Z)$ of superswitching elements is infinite.
\end{prp}

\begin{proof}
We consider two cases depending on whether the set $\Ic$ of involutions is negligible or not.

(i) If $\Ic$ is negligible, then, in view of formula \eqref{eq:u}, \lemref{lem:sixy} and formula \eqref{eq:i} allow one to use the same argument as in the proof of \lemref{lem:s}.

(ii) If $\Ic$ is not negligible, then we use the fact that
$$
\Sf_{\pm 1}(Z) \cap \Ic = \Sf_1(Z) \cap \Ic \;.
$$
By \eqref{eq:us}
$$
\Sf_1(Z) \cap \Ic = \Ic \sm \bigcup_{x,y\in Z\sm\{e\}} \Sc_{x\to y} \;,
$$
and \lemref{lem:ab} implies that there are infinitely many superswitching involutions.
\end{proof}

\begin{rem}
A simple example of an infinite group with a non-negligible set of involutions~$\Ic$ is provided by the infinite dihedral group $\Dc_\infty$. If one realizes it as the affine group of the integer line, i.e., as the ``$ax+b$ group'' with $a=\pm 1$ and $b\in\ZZ$, then non-trivial involutions are precisely the elements $(-1,b)$, whence in this case $m(\Ic)=1/2$ for any invariant mean on $\Dc_\infty$, so that $\Ic$ is non-negligible (see \remref{rem:dens}). It would be interesting to have more examples of groups with this property, in particular, among non-amenable or ICC groups.
\end{rem}

As an immediate consequence of \prpref{prp:ss} we obtain

\begin{cor} \label{cor:icclad}
For any ICC group $G$, any gauge function $\la$, and any sequence of finite filling sets $\Ab=(A_0,A_1,\dots)$ there exists a sequence of finite symmetric sets $\Sib=(\Si_1,\Si_2,\dots)$
such that the scale $\La=(\la,\Sib,\Ab)$ is a ladder.
\end{cor}

\begin{proof}
By \lemref{lem:5la} and \prpref{prp:ss} the desired sequence $\Sib=(\Si_1,\Si_2,\dots)$ can be constructed recursively, by choosing at each step $n$ a symmetric set $\Si_n$ which is a switcher for the finite set
$$
\De^{5\la(n)} = \left( \bigcup_{i=1}^{n-1} \Si_i^{\pm 1} \; \cup \; \bigcup_{i=0}^{n-1} A_i^{\pm 1} \right)^{5\la(n)} \;.
$$
\end{proof}

By applying \corref{cor:icclad} to the situation when the sets $A_i$ are all symmetric, contain at most 2 elements, and their union is the whole group $G$ we obtain

\begin{cor} \label{cor:icclad2}
For any ICC group and any gauge function $\la$ there is a ladder $\La=(\la,\Sib,\Ab)$ such that the sets $\Si_i,A_i$ are all symmetric, contain at most 2 elements each, and their union is the whole group $G$.
\end{cor}

\section{Stabilizers of the boundary action} \label{sec:final}

\subsection{Action on the Poisson boundary} \label{sec:act}

Let $H\ns G$ be a normal subgroup of a countable group $G$, and let $\ov G=G/H$ be the corresponding quotient group. Given a probability measure $\mu$ on~$G$, we shall denote by $\ov\mu$ the probability measure on $\ov G$ which is the image of the measure~$\mu$ under the quotient map $G\to\ov G$. The lift of any $\ov\mu$-harmonic function from $\ov G$ to $G$ is obviously $\mu$-harmonic and (left) $H$-invariant, and, conversely, any $H$-invariant $\mu$-harmonic function on $G$ corresponds to a $\ov\mu$-harmonic function on $\ov G$. In view of the Poisson formula \eqref{eq:poisson}, this correspondence implies

\begin{prp}[{\cite[Theorem 2.1.4]{Kaimanovich95}}] \label{prp:h}
Given a random walk $(G,\mu)$ and a normal subgroup $H\ns G$, the Poisson boundary $\p_{\ov\mu}\ov G$ of the quotient random walk $\bigl(\ov G,\ov\mu\bigr)$ is the space of ergodic components of the action of the normal subgroup $H$ on the Poisson boundary $\p_\mu G$ of the original random walk $(G,\mu)$, and, moreover, the harmonic measure $\ov\nu$ on $\p_{\ov\mu}\ov G$ is the image of the harmonic measure $\nu$ on $\p_\mu G$ under the quotient map $\p_\mu G \to \p_{\ov\mu}\ov G$.
\end{prp}

\begin{dfn}
The \textsf{pointwise stabilizer of the Poisson boundary}
%$$
\begin{equation} \label{eq:stab}
\Stab \p_\mu G = \bigl\{g\in G: g\ze=\ze \;\text{for}\;\nnu\text{-a.e.}\; \ze\in\p_\mu G \bigr\}
\end{equation}
%$$
consists of all elements of the group $G$ which act trivially on $\p_\mu G$. It is a normal subgroup of $G$, and, in view of the Poisson formula, it coincides with the \textsf{group of (left) periods of bounded $\mu$-harmonic functions} on $G$ (or, just the \textsf{group of (left) $\mu$-periods}, in short), i.e.,
$$
\Stab \p_\mu G = \bigl\{g\in G: f(gx)=f(x) \quad \forall f\in H^\infty(G,\mu),\; x\in G \bigr\} \;.
$$
\end{dfn}

The quotient map $\p_\mu G \to \p_{\ov\mu}\ov G$ from the Poisson boundary of the original random walk to the Poisson boundary of the quotient random walk on $\ov G=G/\Stab \p_\mu G$ is then an isomorphism by \prpref{prp:h}.

\begin{rem} \label{rem:ss}
If $\sgr\mu\,(\sgr\mu)^{-1}\neq G$, then any $g\in G$ with $\sgr\mu\cap g\sgr\mu=\vn$ has the property that the harmonic measure $\nu$ and its translate $g\nu$ are mutually singular (e.g., see \cite[Lemma 2.9]{Kaimanovich92}), and therefore the action of $g$ on the Poisson boundary $\p_\mu G$ is non-trivial. Thus,
$$
\Stab\p_\mu G \subset \sgr\mu\,(\sgr\mu)^{-1} \;.
$$
\end{rem}

\begin{rem}
The following dichotomy holds for any ergodic measure class preserving action of a countable group $G$ on a Lebesgue measure space $(X,m)$: either
\begin{itemize}
\item[(i)]
there is a normal subgroup $H\ns G$ such that $\Stab x = H$ for almost every point $x\in X$; or
\item[(ii)]
for almost every point $x\in X$ the stabilizer $\Stab x\subset G$ is not a normal subgroup.
\end{itemize}
Indeed, it is not difficult to see that the map $x\mapsto\Stab x$ is measurable and equivariant (see Vershik \cite[the remark after Definition 2]{Vershik12}), and it is constant along the orbit $Gx$ of a point $x$ if and only if $\Stab x$ is normal. Therefore, \prpref{prp:h} implies that given a random walk $(G,\mu)$ either
\begin{itemize}
\item[(i)]
the pointwise stabilizer $H=\Stab \p_\mu G$ has the property that the action of the associated quotient group $\ov G=G/H$ on the Poisson boundary $\p_{\ov\mu}\ov G\cong\p_\mu G$ of the quotient random walk $\bigl(\ov G,\ov\mu\bigr)$ is \textup{(mod 0)} free; or
\item[(ii)]
for almost every point of the Poisson boundary $\p_\mu G$ its stabilizer is not a normal subgroup.
\end{itemize}
Examples of type (ii) random walks are readily provided by continuous groups (for instance, the ``$ax+b$'' group). However, at the moment we are not aware of any type (ii) random walks on countable groups.
\end{rem}

\subsection{Hyper-FC-centre} \label{sec:icc}

We have to remind the definitions of the FC-centre and of the hyper-FC-centre of a group (for more background see, for instance, Robinson \cite[Section 4.3]{Robinson72} and the references therein).

An element $g$ of a group $G$ is central if and only if its conjugacy class $\Cc_g$ \eqref{eq:conj} consists just of $g$ itself. An \textsf{FC-element} is $g\in G$ such that its conjugacy class is finite, or, equivalently, its \textsf{centralizer}
$$
\C(g) = \{h\in G: h^{-1}gh=g \} = \{h\in G: gh=hg \}
$$
has a finite index in $G$. The collection $\FC(G)$ of all FC-elements of a group $G$, i.e., the union of all finite conjugacy classes, is called the \textsf{FC-centre} of $G$; it is a normal (and even characteristic) subgroup of $G$. A group $G$ is said to be a \textsf{group with finite conjugacy classes} (or, an \textsf{FC group} in short) if $\FC(G)=G$, i.e., if the conjugacy class of any element is finite. The ICC groups (see the beginning of \secref{sec:sicc}) can be described in these terms as the ones for which $\FC(G)=\{e\}$.

The series
$$
G=G_0 \to G_1 \to \dots \to G_n \to G_{n+1} \to \dots \;,
$$
where each consecutive group $G_{n+1}=G_n/\FC(G_n)$ is the quotient of the preceding one $G_n$ by its FC-centre $\FC(G_n)$, gives rise to the \textsf{upper FC-central series} of the group $G$
%$$
\begin{equation} \label{eq:fcseries}
\{e\} = \FC_0(G) \ns \FC_1(G) \ns \dots \ns \FC_n(G) \ns \FC_{n+1}(G) \ns \dots,
\end{equation}
%$$
which consists of the kernels $\FC_n(G)$ of the quotient homomorphisms $G\to G_n$, so that
%$$
\begin{equation} \label{eq:fc}
G/\FC_{n+1}(G) = \Bigl[ G/\FC_n(G) \Bigr] \; \Big/ \, \FC \bigl[ G/\FC_n(G) \bigr] \;.
\end{equation}
%$$

The series \eqref{eq:fcseries} extends transfinitely in the usual way: for a successor ordinal $\al+1$ the corresponding entry $\FC_{\al+1}(G)$ is obtained from $\FC_\al(G)$ by using \eqref{eq:fc}, whereas for a limit ordinal $\al$ one puts
%$$
\begin{equation} \label{eq:albe}
\FC_\al(G) = \bigcup_{\be <\al} \FC_\be(G) \;.
\end{equation}
%$$
The transfinite upper FC-central series eventually stabilizes, the limit
$$
\FC_{\lim}(G) = \bigcup_\al \FC_\al(G) \;,
$$
where the union is taken over all ordinals $\al$, is called the \textsf{hyper-FC-centre} of the group~$G$, and its elements are called \textsf{hyper-FC-central}. In other words, $\FC_{\lim}(G)$ is the minimal among the normal subgroups $H$ of $G$ with the property that the FC-centre of the quotient $G/H$ is trivial, i.e., that $G/H$ is an ICC group. If $\FC_{\lim}(G)=G$, then the group $G$ is called
\textsf{hyper-FC-central}. A finitely generated hyper-FC-central group is virtually nilpotent, i.e., it contains a nilpotent subgroup of finite index, and its upper FC-central series is finite.

\subsection{Boundary action of the hyper-FC-centre}

The centre $\C(G)$ is contained in the group of $\mu$-periods $\Stab\p_\mu G$ for any non-degenerate probability measure $\mu$ on a countable group $G$ (i.e., such that $\sgr\mu=G$, see \secref{sec:rw}). This observation was apparently first made by Dynkin~-- Malyutov \cite[Lemma 1]{Dynkin-Malutov61} (their argument is actually quite close to the one earlier used by Heilbronn \cite[Theorem~5]{Heilbronn49} to prove the absence of non-constant bounded harmonic functions for the simple random walk on $\ZZ^d$). They applied it as the induction step in the proof that the Poisson boundary for the non-degenerate random walks on countable nilpotent groups is trivial \cite[Theorem~1]{Dynkin-Malutov61}, or, in a somewhat more general formulation, that any entry of the upper central series of the group $G$ is contained in $\Stab\p_\mu G$ under the same non-degeneracy condition on $\mu$.

Similar results were later obtained by a number of authors: Margulis \cite{Margulis66}, Furstenberg \cite[Theorem~11.2]{Furstenberg73},  Lyons~-- Sullivan \cite[Section 4]{Lyons-Sullivan84}, Lin \cite[Lemma 3.3 and Theorems~3.4,~3.9]{Lin87}, Jaworski \cite[Theorem 3.9]{Jaworski94}, Kaimanovich \cite[Theorems 4.1.4 and 5.2.7]{Kaimanovich95}. In some of them the induction was made transfinite to imply that the $\om$-centre \cite{Lyons-Sullivan84} or the hypercentre \cite{Lin87, Kaimanovich95} of the group $G$ is contained in the group of periods. Actually, \cite[Theorems 4.1.5]{Kaimanovich95} contains a somewhat stronger result, in which the hypercentre is replaced with the minimal normal subgroup $H\ns G$ with the property that the quotient $G/H$ is centreless (this is a description of the hypercentre of $G$) \emph{and} has no finite normal subgroups. The class of groups arising in this way is intermediate between the classes of hypercentral and hyper-FC-central groups (for instance, the infinite dihedral group is FC-central, but it is centreless and has no finite normal subgroups).

Lin -- Zaidenberg \cite[Lemma 2.4]{Lin-Zaidenberg98} noticed that, in the setup of bounded harmonic or holomorphic functions on covering manifolds, the induction step can be done for the FC-centre $\FC(G)$ rather than just for the usual centre $\C(G)$ of the deck transformations group $G$. It then implies, through transfinite induction, that the group of periods of these functions contains the whole hyper-FC-centre $\FC_{\lim}(G)$ \cite[Corollary 2.5]{Lin-Zaidenberg98}.

The periods of harmonic functions of random walks on topological (especially, Lie) groups were studied in great detail in the 60s-70s in the works of Furstenberg \cite{Furstenberg63}, Azencott \cite{Azencott70}, Guivarc'h \cite{Guivarch73} and Raugi \cite{Raugi77}. However, possible specializations to countable groups were usually left out, as, for instance, the following result, which is essentially contained in the work of Azencott:

\begin{lem}[Jaworski {\cite[Lemma 4.7]{Jaworski04}}] \label{lem:fc}
If $\mu$ is a probability measure on a countable group $G$ such that
%$$
\begin{equation} \label{eq:ss}
\sgr\mu\,(\sgr\mu)^{-1}=G
\end{equation}
%$$
(in particular, if the measure $\mu$ is non-degenerate), then
$$
\FC(G) \subset \Stab\p_\mu G \;.
$$
\end{lem}

Jaworski deduced it from Azencott's \cite[Th\'eor\`eme IV.1]{Azencott70} or from his own 0-2 law \cite[Proposition 3.1]{Jaworski94} which imply that under condition \eqref{eq:ss} $\FC(G)$ is contained in the group of the \emph{right} $\mu$-periods
$$
\Stab\nu = \{g\in G: g\nu=\nu \} \;.
$$
After that he used the same argument as in \cite[Lemme~IV.7]{Azencott70} in order to pass to the group of \emph{left} $\mu$-periods, i.e., to the pointwise stabilizer  $\Stab\p_\mu G$ of the Poisson boundary.

A standard transfinite induction (almost \emph{verbatim} following the proof of \cite[Corollary 2.5]{Lin-Zaidenberg98}) then turns \lemref{lem:fc} into

\begin{prp} \label{prp:fc}
If $\mu$ is a probability measure on a countable group $G$ which satisfies the condition
$$
\sgr\mu\,(\sgr\mu)^{-1}=G
$$
(in particular, if the measure $\mu$ is non-degenerate), then
$$
\FC_{\lim}(G) \subset \Stab\p_\mu G \;,
$$
i.e., the action of the hyper-FC-centre on the Poisson boundary $\p_\mu G$ is trivial.
\end{prp}

\begin{proof}
Let $G_\al=G/\FC_\al(G)$ and $G_{\al+1}=G/\FC_{\al+1}(G)$ be the quotient groups of $G$ determined by the entries of the upper FC-central series indexed by an ordinal $\al$ and its successor $\al+1$, respectively. Then the homomorphisms
$$
G \to G_\al \to G_{\al+1}
$$
give rise to the corresponding quotient maps between the Poisson boundaries
$$
\p_\mu G \to \p_{\mu_\al} G_\al \to \p_{\mu_{\al+1}} G_{\al+1} \;,
$$
where $\mu_\al$ and $\mu_{\al+1}$ are the respective quotient measures. If $\FC_\al(G)\subset \Stab\p_\mu G$, then the map $\p_\mu G \to \p_{\mu_\al} G_\al$ is an isomorphism, whereas the second map $\p_{\mu_\al} G_\al \to \p_{\mu_{\al+1}} G_{\al+1}$ is an isomorphism by \lemref{lem:fc} applied to
the random walk $(G_\al,\mu_\al)$ (because condition \eqref{eq:ss} is preserved when passing to quotient random walks). Therefore, the composition $\p_\mu G \to \p_{\mu_{\al+1}} G_{\al+1}$ is also an isomorphism, which means that $\FC_{\al+1}(G)\subset \Stab\p_\mu G$.

On the other hand, if $\al$ is a limit ordinal such that $\FC_\be(G)\subset \Stab\p_\mu G$ for all $\be<\al$ then by \eqref{eq:albe} $\FC_\al(G)$ is also contained in $\Stab\p_\mu G$.
\end{proof}

\begin{rem} \label{rem:hss}
If $G$ is a hyper-FC-central group, i.e., $G=\FC_{\lim}(G)$, then $SS^{-1}=\gr(S)$ for any subsemigroup $S\subset G$. For instance, it follows from the fact that $SS^{-1}=\gr(S)$ for any subsemigroup in a group of subexponential growth (see \cite[Theorem~2.1.15]{Kaimanovich85t} or \cite[Theorem 5.6]{Jaworski94a}), whereas any finitely generated subgroup of a hyper-FC-central group is virtually nilpotent, hence of polynomial growth. Therefore, if $\mu$ is an irreducible probability measure on a hyper-FC-central group $G$, i.e., if $\gr(\mu)=G$, then condition \eqref{eq:ss} is automatically satisfied, so that \prpref{prp:fc} implies the theorem of Jaworski \cite[Theorem 4.8]{Jaworski04} on the triviality of the Poisson boundary $\p_\mu G$ for any irreducible probability measure $\mu$ on a hyper-FC-central group $G$.
\end{rem}

\subsection{Examples of non-trivial boundary action of the centre} \label{sec:ex}

If condition \eqref{eq:ss} is not satisfied, then the action of any $g\in G\sm\sgr\mu\,(\sgr\mu)^{-1}$ on the Poisson boundary $\p_\mu G$ is non-trivial (see \remref{rem:ss}). On the other hand,
$$
SS^{-1}\cap H=\{e\}
$$
for a subsemigroup $S\subset G$ and a normal subgroup $H\ns G$ if and only if the restriction of the quotient homomorphism $\pi:G\to G/H$ to $S$ is a bijection. Thus, we obtain

\begin{lem} \label{lem:semifree}
If a group $G$ is such that the centre $\C(G)$ is non-trivial, and there is a finite generating set $K$ with the property that the image of the semigroup $S=\sgr(K)$ under the quotient homomorphism $\pi:G\to G/C(G)$ is the free semigroup with the free generating set $\pi(K)$, then any probability measure $\mu$ with $\supp\mu=K$ (more generally, any $\mu$ with $\sgr\mu=S$) is irreducible and has the property that any non-identity element of $\C(G)$ acts non-trivially on the Poisson boundary $\p_\mu G$.
\end{lem}

\begin{cor}
There exist a finitely generated group $G$, a probability measure $\mu$ on $G$, and a central element $c\in\C(G)$ which acts non-trivially on the Poisson boundary of $(G,\mu)$. Moreover, $G$ can be chosen to have an infinite centre $\C(G)$ and such that all non-identity elements of $\C(G)$ act non-trivially on the Poisson boundary.
\end{cor}

Indeed, a well-known family of groups satisfying the conditions of \lemref{lem:semifree} was constructed by Hall in his proof of \cite[Theorem 6]{Hall54} which states that any countable abelian group can be realized as the centre of a finitely generated solvable group. He exhibited a 2-generated solvable group $G$ such that its centre $\C(G)$ is an infinitely generated free abelian group, and the quotient $G/\C(G)$ is the wreath product $\ZZ\wr\ZZ=\ZZ\sd \sum_{\ZZ} \ZZ$ \cite[Theorem~7]{Hall54} (also see the proof of Hall's theorem in \cite[14.1.1]{Robinson96}). The group~$G$ itself and its quotients by various proper subgroups of $\C(G)$ then provide the examples Hall was looking for, and all these groups satisfy the conditions of \lemref{lem:semifree}.

The fact that the matrix
{\tiny$\begin{pmatrix} \mi 1 & \!\!\!0 \\ \; 0 & \!\!\!\!\!\mi 1 \end{pmatrix}$}
from the centre of the group $G=SL(2,\RR)$ acts non-trivially on the Poisson boundary $\p_\mu G$ for any absolutely continuous measure $\mu$ whose support is in the set of matrices with non-negative entries was observed already by Furstenberg \cite[pp.\ 378-379]{Furstenberg63}. This example can be easily recast for $SL(2,\ZZ)$ and other discrete subgroups of $SL(2,\RR)$ (or of more general semi-simple Lie groups). However, we are not aware of any prior examples of this phenomenon for groups with an infinite centre.

\begin{rem}
The non-trivial action of the centre on the Poisson boundary is always \textsf{discontinual} in the sense that the ergodic components of the action are its orbits (see \cite{Kaimanovich10} for a discussion of this notion). This should be contrasted with the fact that the action of any normal subgroup on the Poisson boundary is conservative in the situation when the step distribution $\mu$ is non-degenerate \cite[Theorem 3.3.3]{Kaimanovich95}. Actually, one can show that the action of the hyper-FC-centre on the Poisson boundary is also always discontinual (if it is trivial, then the ergodic components are just the single points of the Poisson boundary). We shall return to this subject elsewhere.
\end{rem}

\subsection{The Poisson boundary of ICC groups (Theorem A)}

We can now formulate and prove our principal result:

\begin{thm}[\;=\;Theorem A] \label{thm:main}
For any countable ICC group $G$ there exist a probability measure $\mu$ on~$G$ (which can be chosen to be symmetric, non-degenerate and of finite entropy) and a locally finite forest $\Fc$ with the vertex set $G$ such that:
\begin{enumerate}[{\rm (i)}]
\item
Almost all sample paths of the random walk $(G,\mu)$ converge to the boundary $\p\Fc$ of the forest $\Fc$,
\item
Almost every sample path visits all but a finite number of points of the geodesic ray in $\Fc$ determined by its limit point,
\item
The Poisson boundary $\p_\mu G$ coincides with the boundary $\p\Fc$ endowed with the family of the resulting hitting distributions,
\item
The action of the group $G$ on the Poisson boundary $\p_\mu G$ is free \textup{(mod 0)}, i.e., the stabilizers of almost all boundary points are trivial.
\end{enumerate}
\end{thm}

\begin{proof}
Let us fix a probability measure $p=(p_j)$ on $\ZZ_+$ which satisfies the eventual record simplicity condition of \lemref{lem:bsw}
$$
\sum_j \left( \frac{p_j}{p_j + p_{j+1} + \dots} \right)^2<\infty
$$
(for instance, any measure with polynomial decay) and has a finite entropy. Let $\Phi,\Psi:\NN\to\NN$ be the corresponding functions from \lemref{lem:psi} determined by the distribution~$p$.

Let $(\Sib,\Ab)$ be a $\Phi$-ladder on $G$ with symmetric entries of cardinality $\le 2$ whose union is the whole group $G$, which exists by \corref{cor:icclad2}. Let us take a sequence of coefficients $0<\al_i<1$ such that $\sum_i \al_i<\infty$, and define a symmetric probability measure $\mu$ with $\supp\mu=G$ by putting
$$
\begin{aligned}
\mu(A_0) &= p_0 \;,\\
\mu(\Si_i \cup A_i) &= p_i \;, \quad &i\ge 1 \;,\\
\mu(A_i\sm \Si_i) &\le \al_i p_i \;, \quad &i \ge 1 \;.
\end{aligned}
$$
Note that the finiteness of the entropy of the distribution $p$ implies that the entropy of $\mu$ is also finite. Then \thmref{thm:poi} implies the claim.
\end{proof}

As we have seen in \prpref{prp:fc}, for any non-degenerate probability measure $\mu$ on a countable group $G$ the action of the hyper-FC-centre $\FC_{\lim}(G)$ on the Poisson boundary is trivial, i.e., $\FC_{\lim}(G) \subset \Stab\p_\mu G$. \thmref{thm:main} immediately implies that $\FC_{\lim}(G)$ can, in fact, be realized as $\Stab\p_\mu G$, even stronger:

\begin{cor} \label{cor:main}
For any countable group $G$ there exists a non-degenerate probability measure on $G$ (which can be chosen to be symmetric and to have a finite entropy) such that for almost every point of the Poisson boundary $\p_\mu G$ its stabilizer is the hyper-FC-centre $\FC_{\lim}(G)$ of the group $G$.
\end{cor}

\begin{proof}
Let $\ov G=G/\FC_{\lim}(G)$ be the quotient of the group $G$ by it hyper-FC-centre $\FC_{\lim(G)}$. By \thmref{thm:main} there exists a finite entropy symmetric non-degenerate probability measure $\ov\mu$ on $\ov G$ such that the action of the group $\ov G$ on its Poisson boundary $\p_{\ov\mu}\ov G$ is free (mod 0). Let $\mu$ be a lift of $\ov\mu$ to $G$, which can obviously be chosen non-degenerate, symmetric and to have a finite entropy. Then by \prpref{prp:h} $\p_{\ov\mu}\ov G$ is the space of ergodic components of the action of $\FC_{\lim(G)}$ on the Poisson boundary $\p_\mu G$ of the lifted random walk, which
is trivial by \prpref{prp:fc}. Therefore, the action of $\ov G$ on $\p_{\ov\mu}\ov G$ is the same as the action of $G$ on $\p_\mu G$ factorized through the quotient homomorphism $G\to \ov G$. However, the former action as free by \thmref{thm:main}, whence the claim.
\end{proof}

\subsection{Characterization of the stabilizers of the Poisson boundary (Theorem B)}

As we have just seen (\prpref{prp:fc} and \corref{cor:main}), the inclusion
%$$
\begin{equation} \label{eq:inclh}
\FC_{\lim}(G) \subset \Stab\p_\mu G
\end{equation}
%$$
holds for any non-degenerate probability measure $\mu$ on a countable group $G$, and there always are measures for which \eqref{eq:inclh} holds as equality.

There is yet another necessary condition on the subgroup $\Stab\p_\mu G$, which follows from the fact that the action of a countable group $G$ on its Poisson boundary is \emph{amenable}, see Zimmer \cite[Corollary 5.3]{Zimmer78} (or an entirely probabilistic argument in \cite[Theorem~5.2]{Kaimanovich05a}). On the other hand, the point stabilizers of an amenable action are almost surely amenable by Elliott -- Giordano \cite[Proposition~1.2]{Elliott-Giordano93}, whence $\Stab\p_\mu G$ is always amenable.

It is well-known that there is the maximal one among the normal amenable subgroups of a countable group $G$. It is called the \textsf{amenable radical} of $G$, and we shall denote it by $\RA(G)$ (see Zimmer \cite[Proposition 4.1.12]{Zimmer84} or Furman \cite[Proposition 7]{Furman03} for its existence). If $G$ is non-amenable, then the quotient $G/\RA(G)$ is a non-amenable ICC group. In these terms the above fact can be reformulated as the inclusion
%$$
\begin{equation} \label{eq:incla}
\Stab\p_\mu G \subset \RA(G) \;.
\end{equation}
%$$

The following result is a generalization of the conjecture of Furstenberg \cite[Section~9]{Furstenberg73} (proved in \cite[Theorem 1.10]{Rosenblatt81} and \cite[Theorem~4.3]{Kaimanovich-Vershik83}) that on any amenable group $G$ (i.e., on one with $G=\RA(G)$) there exists a probability measure $\mu$ with a non-trivial Poisson boundary.

\begin{prp}[{\cite[Theorem 2]{Kaimanovich02a}}] \label{prp:ext}
If $H$ is an amenable normal subgroup of a countable group $G$, then for any probability measure $\ov\mu$ on the quotient group $\ov G=G/H$ there exists a lift $\mu$ of the measure $\ov\mu$ to $G$ such that the action of $H$ on the Poisson boundary $\p_\mu G$ is trivial.
\end{prp}

Additionally, if the measure $\ov\mu$ on $\ov G$ in \prpref{prp:ext} is non-degenerate, then the measure $\mu$ in the argument from \cite{Kaimanovich02a} can be chosen to be non-degenerate as well. Applied to the amenable radical $\RA(G)$, \prpref{prp:ext} implies then that for any group~$G$ there exists a non-degenerate probability measure $\mu$ on $G$, for which inclusion \eqref{eq:incla} holds as equality.

Thus, by \eqref{eq:inclh} and \eqref{eq:incla}, the pointwise stabilizer $\Stab\p_\mu G$ of the Poisson boundary $\p_\mu G$ of a non-degenerate probability measure $\mu$ on a countable group $G$ is always sandwiched between the hyper-FC-centre and the amenable radical of $G$ as
$$
\FC_{\lim}(G) \subset \Stab\p_\mu G \subset \RA(G) \;,
$$
and, as we have just explained, both $\FC_{\lim}(G)$ and $\RA(G)$ can be realized as $\Stab\p_\mu G$ for an appropriate choice of $\mu$. A combination of \thmref{thm:main} and \prpref{prp:ext} in fact allows us to give a complete description of the subgroups $H\subset G$ which can appear as $H=\Stab\p_\mu G$ for a certain non-degenerate measure $\mu$ on $G$.

\begin{thm}[\;=\;Theorem B] \label{thm:stab}
For any countable group $G$ the following conditions on a subgroup $H\subset G$ are equivalent:
\begin{enumerate}[{\rm (i)}]
\item
there exists a non-degenerate probability measure $\mu$ on $G$ such that the stabilizer of almost every point of the Poisson boundary $\p_\mu G$ is $H$;
\item
there exists a non-degenerate probability measure $\mu$ on $G$ such that the pointwise stabilizer of the Poisson boundary $\p_\mu G$ is $H$ (i.e., $H$ is the intersection of almost all point stabilizers);
\item
$H$ is an amenable normal subgroup of $G$ with the property that the quotient $G/H$ is either an ICC group or the trivial group.
\end{enumerate}
\end{thm}

\begin{proof}
(i)$\implies$(ii) is obvious.

(ii)$\implies$(iii) The pointwise stabilizer $\Stab\p_\mu G$ is clearly a normal subgroup of $G$, its amenability is inclusion \eqref{eq:incla}, and the fact that $G/\Stab\p_\mu G$ is an ICC group follows from \prpref{prp:h} and \prpref{prp:fc}.

(iii)$\implies$(i) By \thmref{thm:main} there is a non-degenerate probability measure $\ov\mu$ on the quotient group $\ov G=G/H$ such that the action of $\ov G$ on the Poisson boundary $\p_{\ov\mu}\ov G$ is (mod~0) free. Let $\mu$ be a lift of $\ov\mu$ to $G$ provided by \prpref{prp:ext}. Then by \prpref{prp:h} it satisfies the conditions described in (i).
\end{proof}

\begin{rem}
Property (iii) implies that
%$$
\begin{equation} \label{eq:ar}
\FC_{\lim}(G) \subset H \subset \RA(G) \;,
\end{equation}
%$$
However, the ICC property need not be inherited when passing to quotient groups, and there are numerous examples of normal subgroups $H\ns G$, for which inclusions \eqref{eq:ar} are satisfied, but $G/H$ is \emph{not} an ICC group.

The simplest situation is when $G$ is an amenable ICC group, so that $\FC_{\lim}(G)=\{e\}$ and $\RA(G)=G$. Then $G$ may well have non-ICC quotients. For instance, the Baumslag -- Solitar group $\BS(1,2)$ is an amenable ICC group, and it admits the infinite cyclic group~$\ZZ$ as its quotient. More generally, an example of this kind can be provided by any non virtually nilpotent finitely generated elementary amenable group. Indeed, on one hand the quotient of such a group by its hyper-FC-centre is a finitely generated elementary amenable ICC group. On the other hand, any finitely generated infinite elementary amenable group admits an infinite virtually cyclic quotient, see Hillman \cite[Lemma 1 of Section I]{Hillman94}.
\end{rem}

\begin{rem}
We remind that the results of \cite{Furstenberg73,Kaimanovich-Vershik83,Rosenblatt81,Frisch-Hartman-Tamuz-Ferdowsi18} provide a complete classification of countable groups with respect to the existence of irreducible measures with the trivial Poisson boundary ($\equiv$ Liouville measures):
\begin{enumerate}[{\rm (i)}]
\item
hyper-FC-central groups (every measure is Liouville),
\item
amenable non-hyper-FC-central groups (there exist both Liouville and non-Lioville measures),
\item
non-amenable groups (every measure is non-Liouville).
\end{enumerate}
However, the questions about the existence of Liouville or non-Liouville measures of a particular kind (e.g., finite entropy, finite first moment, finite support, symmetric, etc.) on amenable non-hyper-FC-central groups are quite difficult. The answers to these questions are often non-trivial, see, for instance, the examples of
\begin{enumerate}[$\bullet$]
\item
exponential growth groups, on which all finitely supported symmetric measures are Liouville, but still there are (non-symmetric) non-Liouville finitely supported measures \cite{Kaimanovich-Vershik83};
\item
exponential growth groups, on which any finitely supported measure $\mu$ (no matter, symmetric or not, and not necessarily irreducible) is Liouville with respect to the group generated by $\supp\mu$ \cite{Bartholdi-Erschler17};
\item
intermediate growth groups with finite entropy non-Liouville measures \cite{Erschler04a,Erschler-Zheng18p};
\item
amenable groups, on which any finite entropy non-degenerate measure is non-Liouville \cite{Erschler04}.
\end{enumerate}
In the same way, in view of \thmref{thm:stab} it would be interesting to describe the normal subgroups $H\ns G$ of a given group $G$ which can be realized as $H=\Stab\p_\mu G$ for non-degenerate measures on $G$ of a particular kind (cf. above). For example, the aforementioned result from \cite{Erschler04} implies that the measures arising in \prpref{prp:ext} need not have finite entropy.
\end{rem}

\section*{References}

\vskip -.5cm

\let\oldaddcontentsline\addcontentsline% Store \addcontentsline
\renewcommand{\addcontentsline}[3]{}% Make \addcontentsline a no-op

\bibliographystyle{amsalpha}
%\bibliography{C:/S/MyTEX/mine}
\bibliography{erschler-kaimanovich.bbl}

\providecommand{\bysame}{\leavevmode\hbox to3em{\hrulefill}\thinspace}
\providecommand{\MR}{\relax\ifhmode\unskip\space\fi MR }
\providecommand{\MRhref}[2]{%
  \href{http://www.ams.org/mathscinet-getitem?mr=#1}{#2}
}
\providecommand{\href}[2]{#2}
\begin{thebibliography}{AAMBV16}

\bibitem[AAMBV16]{Amir-Angel-MatteBon-Virag16}
Gideon Amir, Omer Angel, Nicol\'{a}s Matte~Bon, and B\'{a}lint Vir\'{a}g,
  \emph{The {L}iouville property for groups acting on rooted trees}, Ann. Inst.
  Henri Poincar\'{e} Probab. Stat. \textbf{52} (2016), no.~4, 1763--1783.
  \MR{3573294}

\bibitem[ABN98]{Arnold-Balakrishnan-Nagaraja98}
Barry~C. Arnold, N.~Balakrishnan, and H.~N. Nagaraja, \emph{Records}, Wiley
  Series in Probability and Statistics: Probability and Statistics, John Wiley
  \& Sons, Inc., New York, 1998, A Wiley-Interscience Publication. \MR{1628157}

\bibitem[Ave74]{Avez74}
Andr{\'e} Avez, \emph{Th\'eor\`eme de {C}hoquet-{D}eny pour les groupes \`a
  croissance non exponentielle}, C. R. Acad. Sci. Paris S\'er. A \textbf{279}
  (1974), 25--28. \MR{0353405 (50 \#5888)}

\bibitem[Aze70]{Azencott70}
Robert Azencott, \emph{Espaces de {P}oisson des groupes localement compacts},
  Springer-Verlag, Berlin, 1970, Lecture Notes in Mathematics, Vol. 148. \MR{58
  \#18748}

\bibitem[BE17]{Bartholdi-Erschler17}
Laurent Bartholdi and Anna Erschler, \emph{Poisson-{F}urstenberg boundary and
  growth of groups}, Probab. Theory Related Fields \textbf{168} (2017),
  no.~1-2, 347--372. \MR{3651055}

\bibitem[BKKO17]{Breuillard-Kalantar-Kennedy-Ozawa17}
Emmanuel Breuillard, Mehrdad Kalantar, Matthew Kennedy, and Narutaka Ozawa,
  \emph{{$C^*$}-simplicity and the unique trace property for discrete groups},
  Publ. Math. Inst. Hautes \'{E}tudes Sci. \textbf{126} (2017), 35--71.
  \MR{3735864}

\bibitem[Bla55]{Blackwell55}
David Blackwell, \emph{On transient {M}arkov processes with a countable number
  of states and stationary transition probabilities}, Ann. Math. Statist.
  \textbf{26} (1955), 654--658. \MR{0075479 (17,754d)}

\bibitem[BN61]{BarndorffNielsen61}
Ole Barndorff-Nielsen, \emph{On the rate of growth of the partial maxima of a
  sequence of independent identically distributed random variables}, Math.
  Scand. \textbf{9} (1961), 383--394. \MR{0139189}

\bibitem[BSW94]{Brands-Steutel-Wilms94}
J.~J. A.~M. Brands, F.~W. Steutel, and R.~J.~G. Wilms, \emph{On the number of
  maxima in a discrete sample}, Statist. Probab. Lett. \textbf{20} (1994),
  no.~3, 209--217. \MR{1294106}

\bibitem[BV05]{Bartholdi-Virag05}
Laurent Bartholdi and B{\'a}lint Vir{\'a}g, \emph{Amenability via random
  walks}, Duke Math. J. \textbf{130} (2005), no.~1, 39--56. \MR{2176547
  (2006h:43001)}

\bibitem[Car72]{Cartier72}
P.~Cartier, \emph{Fonctions harmoniques sur un arbre}, Symposia Mathematica,
  Vol. IX (Convegno di Calcolo delle Probabilit\`a, INDAM, Rome, 1971),
  Academic Press, London, 1972, pp.~203--270. \MR{0353467 (50 \#5950)}

\bibitem[CD60]{Choquet-Deny60}
Gustave Choquet and Jacques Deny, \emph{Sur l'\'equation de convolution {$\mu
  =\mu \ast \sigma $}}, C. R. Acad. Sci. Paris \textbf{250} (1960), 799--801.
  \MR{0119041 (22 \#9808)}

\bibitem[Der80]{Derriennic80}
Yves Derriennic, \emph{Quelques applications du th\'eor\`eme ergodique
  sous-additif}, Conference on Random Walks (Kleebach, 1979) (French),
  Ast\'erisque, vol.~74, Soc. Math. France, Paris, 1980, pp.~183--201, 4.
  \MR{588163 (82e:60013)}

\bibitem[DM61]{Dynkin-Malutov61}
E.~B. Dynkin and M.~B. Maljutov, \emph{Random walk on groups with a finite
  number of generators}, Dokl. Akad. Nauk SSSR \textbf{137} (1961), 1042--1045.
  \MR{0131904 (24 \#A1751)}

\bibitem[Doo59]{Doob59}
J.~L. Doob, \emph{Discrete potential theory and boundaries}, J. Math. Mech.
  \textbf{8} (1959), 433--458; erratum 993. \MR{0107098}

\bibitem[DSW60]{Doob-Snell-Williamson60}
J.~L. Doob, J.~L. Snell, and R.~E. Williamson, \emph{Application of boundary
  theory to sums of independent random variables}, Contributions to probability
  and statistics, Stanford Univ. Press, Stanford, Calif., 1960, pp.~182--197.
  \MR{0120667}

\bibitem[EG93]{Elliott-Giordano93}
G.~A. Elliott and T.~Giordano, \emph{Amenable actions of discrete groups},
  Ergodic Theory Dynam. Systems \textbf{13} (1993), no.~2, 289--318.
  \MR{94i:22023}

\bibitem[Eis09]{Eisenberg09}
Bennett Eisenberg, \emph{The number of players tied for the record}, Statist.
  Probab. Lett. \textbf{79} (2009), no.~3, 283--288. \MR{2493010}

\bibitem[Ers04a]{Erschler04a}
Anna Erschler, \emph{Boundary behavior for groups of subexponential growth},
  Ann. of Math. (2) \textbf{160} (2004), no.~3, 1183--1210. \MR{2144977
  (2006d:20072)}

\bibitem[Ers04b]{Erschler04}
\bysame, \emph{Liouville property for groups and manifolds}, Invent. Math.
  \textbf{155} (2004), no.~1, 55--80. \MR{2 025 301}

\bibitem[Ers10]{Erschler10}
\bysame, \emph{Poisson-{F}urstenberg boundaries, large-scale geometry and
  growth of groups}, Proceedings of the {I}nternational {C}ongress of
  {M}athematicians. {V}olume {II}, Hindustan Book Agency, New Delhi, 2010,
  pp.~681--704. \MR{2827814 (2012h:60016)}

\bibitem[EZ18]{Erschler-Zheng18p}
Anna Erschler and Tianyi Zheng, \emph{Growth of periodic {G}rigorchuk groups},
  arXiv:1802.09077, 2018.

\bibitem[Fel56]{Feller56}
William Feller, \emph{Boundaries induced by non-negative matrices}, Trans.
  Amer. Math. Soc. \textbf{83} (1956), 19--54. \MR{0090927 (19,892a)}

\bibitem[FHTVF18]{Frisch-Hartman-Tamuz-Ferdowsi18}
Joshua Frisch, Yair Hartman, Omer Tamuz, and Pooya Vahidi~Ferdowsi,
  \emph{Choquet-{D}eny groups and the infinite conjugacy class property},
  arXiv:1802.00751, 2018.

\bibitem[FM98]{Farb-Masur98}
Benson Farb and Howard Masur, \emph{Superrigidity and mapping class groups},
  Topology \textbf{37} (1998), no.~6, 1169--1176. \MR{1632912 (99f:57017)}

\bibitem[FTF18]{Frisch-Tamuz-Ferdowsi18}
Joshua Frisch, Omer Tamuz, and Pooya~Vahidi Ferdowsi, \emph{Strong amenability
  and the infinite conjugacy class property}, arXiv:1801.04024, 2018.

\bibitem[Fur63]{Furstenberg63}
Harry Furstenberg, \emph{A {P}oisson formula for semi-simple {L}ie groups},
  Ann. of Math. (2) \textbf{77} (1963), 335--386. \MR{26 \#3820}

\bibitem[Fur67]{Furstenberg67}
\bysame, \emph{Poisson boundaries and envelopes of discrete groups}, Bull.
  Amer. Math. Soc. \textbf{73} (1967), 350--356. \MR{0210812 (35 \#1698)}

\bibitem[Fur71]{Furstenberg71}
\bysame, \emph{Random walks and discrete subgroups of {L}ie groups}, Advances
  in Probability and Related Topics, Vol. 1, Dekker, New York, 1971, pp.~1--63.
  \MR{0284569 (44 \#1794)}

\bibitem[Fur73]{Furstenberg73}
\bysame, \emph{Boundary theory and stochastic processes on homogeneous spaces},
  Harmonic analysis on homogeneous spaces (Proc. Sympos. Pure Math., Vol. XXVI,
  Williams Coll., Williamstown, Mass., 1972), Amer. Math. Soc., Providence,
  R.I., 1973, pp.~193--229. \MR{50 \#4815}

\bibitem[Fur03]{Furman03}
Alex Furman, \emph{On minimal strongly proximal actions of locally compact
  groups}, Israel J. Math. \textbf{136} (2003), 173--187. \MR{1998109}

\bibitem[Gef58]{Geffroy58}
Jean Geffroy, \emph{Contribution \`a la th\'eorie des valeurs extr\^emes},
  Publ. Inst. Statist. Univ. Paris \textbf{7} (1958), no.~3/4, 37--121.
  \MR{0105168}

\bibitem[GKR77]{Guivarch-Keane-Roynette77}
Yves Guivarc'h, Michael Keane, and Bernard Roynette, \emph{Marches al\'eatoires
  sur les groupes de {L}ie}, Lecture Notes in Mathematics, Vol. 624,
  Springer-Verlag, Berlin-New York, 1977. \MR{0517359}

\bibitem[Gui73]{Guivarch73}
Yves Guivarc'h, \emph{Croissance polynomiale et p\'{e}riodes des fonctions
  harmoniques}, Bull. Soc. Math. France \textbf{101} (1973), 333--379.
  \MR{0369608}

\bibitem[Hal54]{Hall54}
P.~Hall, \emph{Finiteness conditions for soluble groups}, Proc. London Math.
  Soc. (3) \textbf{4} (1954), 419--436. \MR{0072873}

\bibitem[Hei49]{Heilbronn49}
H.~A. Heilbronn, \emph{On discrete harmonic functions}, Proc. Cambridge Philos.
  Soc. \textbf{45} (1949), 194--206. \MR{0030051}

\bibitem[Hil94]{Hillman94}
J.~A. Hillman, \emph{The algebraic characterization of geometric
  {$4$}-manifolds}, London Mathematical Society Lecture Note Series, vol. 198,
  Cambridge University Press, Cambridge, 1994. \MR{1275829}

\bibitem[HK17]{Hartman-Kalantar17p}
Yair Hartman and Mehrdad Kalantar, \emph{Stationary {$C^*$}-dynamical systems},
  arXiv:1712.10133, 2017.

\bibitem[Hun60]{Hunt60}
G.~A. Hunt, \emph{Markoff chains and {M}artin boundaries}, Illinois J. Math.
  \textbf{4} (1960), 313--340. \MR{0123364}

\bibitem[Jaw94a]{Jaworski94a}
Wojciech Jaworski, \emph{Exponential boundedness and amenability of open
  subsemigroups of locally compact groups}, Canad. J. Math. \textbf{46} (1994),
  no.~6, 1263--1274. \MR{1304344}

\bibitem[Jaw94b]{Jaworski94}
\bysame, \emph{Strongly approximately transitive group actions, the
  {C}hoquet-{D}eny theorem, and polynomial growth}, Pacific J. Math.
  \textbf{165} (1994), no.~1, 115--129. \MR{1285567 (95e:22015)}

\bibitem[Jaw04]{Jaworski04}
\bysame, \emph{Countable amenable identity excluding groups}, Canad. Math.
  Bull. \textbf{47} (2004), no.~2, 215--228. \MR{2059416}

\bibitem[JLP08]{James-Lyons-Peres08}
Nicholas James, Russell Lyons, and Yuval Peres, \emph{A transient {M}arkov
  chain with finitely many cutpoints}, Probability and statistics: essays in
  honor of {D}avid {A}. {F}reedman, Inst. Math. Stat. (IMS) Collect., vol.~2,
  Inst. Math. Statist., Beachwood, OH, 2008, pp.~24--29. \MR{2459947}

\bibitem[JP96]{James-Peres96}
N.~James and Y.~Peres, \emph{Cutpoints and exchangeable events for random
  walks}, Teor. Veroyatnost. i Primenen. \textbf{41} (1996), no.~4, 854--868.
  \MR{1687097 (2000g:60120)}

\bibitem[Kai85a]{Kaimanovich85t}
V.~A. Kaimanovich, \emph{Boundaries of random walks on discrete groups}, Ph.D.
  thesis, Leningrad State University, 1985, in Russian.

\bibitem[Kai85b]{Kaimanovich85a}
Vadim~A. Kaimanovich, \emph{A complete description of the tail $\sigma$-algebra
  of random walks and related problems}, Teor. Veroyatnost. i Primenen.
  \textbf{30} (1985), no.~1, 207--208.

\bibitem[Kai85c]{Kaimanovich85}
\bysame, \emph{An entropy criterion of maximality for the boundary of random
  walks on discrete groups}, Soviet Math.Dokl. \textbf{31} (1985), 193--197.
  \MR{86m:60025}

\bibitem[Kai91]{Kaimanovich91}
\bysame, \emph{Poisson boundaries of random walks on discrete solvable groups},
  Probability measures on groups, {X} ({O}berwolfach, 1990), Plenum, New York,
  1991, pp.~205--238. \MR{1178986 (94m:60014)}

\bibitem[Kai92]{Kaimanovich92}
\bysame, \emph{Measure-theoretic boundaries of {M}arkov chains, {$0$}-{$2$}
  laws and entropy}, Harmonic analysis and discrete potential theory (Frascati,
  1991), Plenum, New York, 1992, pp.~145--180. \MR{94h:60099}

\bibitem[Kai95]{Kaimanovich95}
\bysame, \emph{The {P}oisson boundary of covering {M}arkov operators}, Israel
  J. Math. \textbf{89} (1995), no.~1-3, 77--134. \MR{96k:60194}

\bibitem[Kai96]{Kaimanovich96}
\bysame, \emph{Boundaries of invariant {M}arkov operators: the identification
  problem}, Ergodic theory of ${\mathbb Z}\sp d$ actions (Warwick, 1993--1994),
  London Math. Soc. Lecture Note Ser., vol. 228, Cambridge Univ. Press,
  Cambridge, 1996, pp.~127--176. \MR{97j:31008}

\bibitem[Kai00]{Kaimanovich00a}
\bysame, \emph{The {P}oisson formula for groups with hyperbolic properties},
  Ann. of Math. (2) \textbf{152} (2000), no.~3, 659--692. \MR{1815698
  (2002d:60064)}

\bibitem[Kai02]{Kaimanovich02a}
\bysame, \emph{The {P}oisson boundary of amenable extensions}, Monatsh. Math.
  \textbf{136} (2002), no.~1, 9--15. \MR{2003e:60013}

\bibitem[Kai05]{Kaimanovich05a}
\bysame, \emph{Amenability and the {L}iouville property}, Israel J. Math.
  \textbf{149} (2005), 45--85, Probability in mathematics. \MR{2191210
  (2007c:43001)}

\bibitem[Kai10]{Kaimanovich10}
\bysame, \emph{Hopf decomposition and horospheric limit sets}, Ann. Acad. Sci.
  Fenn. Math. \textbf{35} (2010), no.~2, 335--350. \MR{2731695}

\bibitem[KK17]{Kalantar-Kennedy17}
Mehrdad Kalantar and Matthew Kennedy, \emph{Boundaries of reduced
  {$C^*$}-algebras of discrete groups}, J. Reine Angew. Math. \textbf{727}
  (2017), 247--267. \MR{3652252}

\bibitem[KM96]{Kaimanovich-Masur96}
Vadim~A. Kaimanovich and Howard Masur, \emph{The {P}oisson boundary of the
  mapping class group}, Invent. Math. \textbf{125} (1996), no.~2, 221--264.
  \MR{1395719 (97m:32033)}

\bibitem[KV83]{Kaimanovich-Vershik83}
V.~A. Kaimanovich and A.~M. Vershik, \emph{Random walks on discrete groups:
  boundary and entropy}, Ann. Probab. \textbf{11} (1983), no.~3, 457--490.
  \MR{85d:60024}

\bibitem[Led01]{Ledrappier01}
Fran{\c{c}}ois Ledrappier, \emph{Some asymptotic properties of random walks on
  free groups}, Topics in probability and Lie groups: boundary theory, CRM
  Proc. Lecture Notes, vol.~28, Amer. Math. Soc., Providence, RI, 2001,
  pp.~117--152. \MR{1832436 (2002g:60116)}

\bibitem[Lin87]{Lin87}
V.~Ya. Lin, \emph{Liouville coverings of complex spaces, and amenable groups},
  Mat. Sb. (N.S.) \textbf{132(174)} (1987), no.~2, 202--224. \MR{882834}

\bibitem[LP15]{Lyons-Peres15}
Russell Lyons and Yuval Peres, \emph{{P}oisson boundaries of lamplighter
  groups: proof of the {K}aimanovich-{V}ershik conjecture}, arXiv:1508.01845,
  2015.

\bibitem[LS84]{Lyons-Sullivan84}
Terry Lyons and Dennis Sullivan, \emph{Function theory, random paths and
  covering spaces}, J. Differential Geom. \textbf{19} (1984), no.~2, 299--323.
  \MR{86b:58130}

\bibitem[LZ98]{Lin-Zaidenberg98}
Vladimir~Ya. Lin and Mikhail Zaidenberg, \emph{Liouville and {C}arath\'{e}odory
  coverings in {R}iemannian and complex geometry}, Voronezh {W}inter
  {M}athematical {S}chools, Amer. Math. Soc. Transl. Ser. 2, vol. 184, Amer.
  Math. Soc., Providence, RI, 1998, pp.~111--130. \MR{1729929}

\bibitem[Mar66]{Margulis66}
G.~A. Margulis, \emph{Positive harmonic functions on nilpotent groups}, Soviet
  Math. Dokl. \textbf{7} (1966), 241--244. \MR{0222217 (36 \#5269)}

\bibitem[MT18]{Maher-Tiozzo18p}
Joseph Maher and Giulio Tiozzo, \emph{Random walks, {WPD} actions, and the
  {C}remona group}, arXiv:1807.10230, 2018.

\bibitem[Neu54]{Neumann54}
B.~H. Neumann, \emph{Groups covered by permutable subsets}, J. London Math.
  Soc. \textbf{29} (1954), 236--248. \MR{0062122}

\bibitem[Nev64]{Neveu64}
Jacques Neveu, \emph{Cha\^\i nes de {M}arkov et th\'eorie du potentiel}, Ann.
  Fac. Sci. Univ. Clermont-Ferrand (1964), no.~24, 37--89. \MR{0386031}

\bibitem[Nev01]{Nevzorov01}
Valery~B. Nevzorov, \emph{Records: mathematical theory}, Translations of
  Mathematical Monographs, vol. 194, American Mathematical Society, Providence,
  RI, 2001, Translated from the Russian manuscript by D. M. Chibisov.
  \MR{1791071}

\bibitem[Qi97]{Qi97}
Yongcheng Qi, \emph{A note on the number of maxima in a discrete sample},
  Statist. Probab. Lett. \textbf{33} (1997), no.~4, 373--377. \MR{1458007}

\bibitem[Rau77]{Raugi77}
Albert Raugi, \emph{Fonctions harmoniques sur les groupes localement compacts
  \`a base d\'enombrable}, Bull. Soc. Math. France. M\'emoires \textbf{54}
  (1977), 5--118.

\bibitem[Rob72]{Robinson72}
Derek J.~S. Robinson, \emph{Finiteness conditions and generalized soluble
  groups. {P}art 1}, Springer-Verlag, New York-Berlin, 1972, Ergebnisse der
  Mathematik und ihrer Grenzgebiete, Band 62. \MR{0332989}

\bibitem[Rob96]{Robinson96}
\bysame, \emph{A course in the theory of groups}, second ed., Graduate Texts in
  Mathematics, vol.~80, Springer-Verlag, New York, 1996. \MR{1357169}

\bibitem[Ros81]{Rosenblatt81}
Joseph Rosenblatt, \emph{Ergodic and mixing random walks on locally compact
  groups}, Math. Ann. \textbf{257} (1981), no.~1, 31--42. \MR{83f:43002}

\bibitem[Spi64]{Spitzer64}
Frank Spitzer, \emph{Principles of random walk}, The University Series in
  Higher Mathematics, D. Van Nostrand Co., Inc., Princeton,
  N.J.-Toronto-London, 1964. \MR{0171290}

\bibitem[Ste92]{Stepanov92}
A.~V. Stepanov, \emph{Limit theorems for weak records}, Teor. Veroyatnost. i
  Primenen. \textbf{37} (1992), no.~3, 586--590. \MR{1214368}

\bibitem[TW98]{Tomkins-Wang98}
R.~J. Tomkins and Hong Wang, \emph{Zero-one laws for large order statistics},
  Order statistics: theory \& methods, Handbook of Statist., vol.~16,
  North-Holland, Amsterdam, 1998, pp.~375--384. \MR{1668752}

\bibitem[Ver73]{Vervaat73}
Wim Vervaat, \emph{Limit theorems for records from discrete distributions},
  Stochastic Processes Appl. \textbf{1} (1973), 317--334. \MR{0362457}

\bibitem[Ver00]{Vershik00}
A.~M. Vershik, \emph{Dynamic theory of growth in groups: entropy, boundaries,
  examples}, Uspekhi Mat. Nauk \textbf{55} (2000), no.~4(334), 59--128.
  \MR{1786730 (2001m:37019)}

\bibitem[Ver12]{Vershik12}
\bysame, \emph{Totally nonfree actions and the infinite symmetric group}, Mosc.
  Math. J. \textbf{12} (2012), no.~1, 193--212, 216. \MR{2952431}

\bibitem[VK79]{Vershik-Kaimanovich79}
A.M. Vershik and V.~A Kaimanovich, \emph{Random walks on groups: boundary,
  entropy, uniform distribution}, Dokl. Akad. Nauk SSSR \textbf{249} (1979),
  no.~1, 15--18. \MR{553972 (81f:60098)}

\bibitem[Woe86]{Woess86}
Wolfgang Woess, \emph{A description of the {M}artin boundary for nearest
  neighbour random walks on free products}, Probability measures on groups,
  {VIII} ({O}berwolfach, 1985), Lecture Notes in Math., vol. 1210, Springer,
  Berlin, 1986, pp.~203--215. \MR{879007}

\bibitem[Woe00]{Woess00}
\bysame, \emph{Random walks on infinite graphs and groups}, Cambridge Tracts in
  Mathematics, vol. 138, Cambridge University Press, Cambridge, 2000.
  \MR{2001k:60006}

\bibitem[Zim78]{Zimmer78}
Robert~J. Zimmer, \emph{Amenable ergodic group actions and an application to
  {P}oisson boundaries of random walks}, J. Functional Analysis \textbf{27}
  (1978), no.~3, 350--372. \MR{57 \#12775}

\bibitem[Zim84]{Zimmer84}
\bysame, \emph{Ergodic theory and semisimple groups}, Monographs in
  Mathematics, vol.~81, Birkh\"auser Verlag, Basel, 1984. \MR{86j:22014}

\end{thebibliography}

\let\addcontentsline\oldaddcontentsline% Restore \addcontentsline

\end{document}